\documentclass[draft,reqno]{amsart}
\usepackage{amsfonts}%
\usepackage{amssymb}
\usepackage{mathrsfs}
\usepackage{cite}
\usepackage{color}
\allowdisplaybreaks
\date{\today}

\theoremstyle{plain}
\newtheorem{thm}{Theorem}[section]

\newtheorem{lem}[thm]{Lemma}
\newtheorem{prop}[thm]{Proposition}
\theoremstyle{definition}

\theoremstyle{remark}

\numberwithin{equation}{section}

\renewcommand{\u}{{\mathbf u}}
\renewcommand{\v}{{\mathbf{v}}}

\renewcommand{\H}{\mathbf{H}}

\newcommand{\w}{{\mathbf w}}
\newcommand{\R}{{\mathbb R}}
\newcommand{\U}{{\mathbf U}}

\newcommand{\W}{{\mathbf W}}

\newcommand{\dv}{{\rm div }}
\newcommand{\cu}{{\rm curl\, }}
\newcommand{\E}{{\mathcal E}}

\begin{document}

\title[Incompressible limit of  compressible non-isentropic MHD equations]
{Incompressible limit of the compressible non-isentropic
  magnetohydrodynamic equations with zero magnetic diffusivity}

\author{Song Jiang}
\address{LCP, Institute of Applied Physics and Computational Mathematics, P.O.
 Box 8009, Beijing 100088, P.R. China}
 \email{jiang@iapcm.ac.cn}

\author{Qiangchang Ju}
\address{Institute of Applied Physics and Computational Mathematics, P.O.
 Box 8009-28, Beijing 100088, P.R. China}
 \email{qiangchang\_ju@yahoo.com}

 \author[Fucai Li]{Fucai Li}
\address{Department of Mathematics, Nanjing University, Nanjing
 210093, P.R. China}
 \email{fli@nju.edu.cn}



\begin{abstract}
We study the incompressible limit of the compressible non-
\linebreak isentropic magnetohydrodynamic equations with zero
magnetic diffusivity and general initial data in the whole space
$\mathbb{R}^d$ ($d=2,3$). We first establish the existence of classic solutions on a time interval independent of
the Mach number. Then, by deriving uniform a priori estimates, we
obtain the convergence of the solution to that of the incompressible
magnetohydrodynamic equations as the Mach number tends to zero.

\end{abstract}

\maketitle

\section{Introduction}

This paper is concerned with the
incompressible limit to the  compressible non-isentropic magnetohydrodynamic (MHD)
equations with zero
magnetic diffusivity and general initial data in the
whole space $\mathbb{R}^d$ ($d=2,3$).

In the study of a highly conducting fluid, for example, the  magnetic fusion, it is
 rational to ignore the magnetic diffusion term in the  MHD equations
 since the  magnetic diffusion coefficient (resistivity coefficient)
  is inversely proportional to the electrical conductivity coefficient,
 see \cite{Fr}. In this situation, the system, describing the motion of the fluid in ${\mathbb R^d}$,
  can be described by the following compressible non-isentropic MHD equations with zero
magnetic diffusivity:
\begin{align}
&\partial_t\rho +\dv(\rho\u)=0, \label{naa} \\
&\partial_t(\rho\u)+\dv\left(\rho\u\otimes\u\right)+ {\nabla p}
  =(\nabla \times \H)\times \H+\dv\Psi, \label{nab} \\
&\partial_t\H-\nabla\times(\u\times\H)=0,\quad
\dv\H=0, \label{nac}\\
&\partial_t\E+\dv\left(\u(\E'+p)\right)
=\dv((\u\times\H)\times\H)+\dv(\u\Psi+\kappa\nabla\theta).
 \label{nad}
\end{align}
Here $\rho $ denotes the density, $\u\in \R^d$ the
velocity, $\H\in \R^d$ the magnetic field, and $\theta$ the
temperature, respectively; $\Psi$ is the viscous stress tensor given by
\begin{equation*}
\Psi=2\mu \mathbb{D}(\u)+\lambda\dv\u \;\mathbf{I}_d
\end{equation*}
with $\mathbb{D}(\u)=(\nabla\u+\nabla\u^\top)/2$, and
 $\mathbf{I}_d$ being the $d\times d$ identity matrix,
and $\nabla \u^\top$ the transpose of the matrix $\nabla \u$; $\E$
is the total energy given by $\E=\E'+|\H|^2/2$ and
$\E'=\rho\left(e+|\u|^2/2 \right)$ with $e$ being the internal
energy, $\rho|\u|^2/2$ the kinetic energy, and $|\H|^2/2$ the
magnetic energy. The viscosity coefficients $\lambda$ and $\mu$ of
the fluid satisfy $2\mu+d\lambda>0$ and $\mu>0$; $\kappa>0$ is the
heat conductivity. For simplicity, we assume that $\mu,\lambda$ and
$\kappa$ are constants. The equations of state $p=p(\rho,\theta)$
and $e=e(\rho,\theta)$ relate the pressure $p$ and the internal
energy $e$ to the density $\rho$ and the temperature $\theta$ of the
flow.

For the smooth solution to the system \eqref{naa}--\eqref{nad},
we can rewrite the total energy equation  \eqref{nad} in the form of
the internal energy. In fact, multiplying  \eqref{nab} by $\u$ and
\eqref{nac} by $\H$, and summing the resulting equations together, we obtain
\begin{align}\label{naaz}
\frac{d}{dt}\Big(\frac{1}{2}\rho|\u|^2
+\frac{1}{2}|\H|^2\Big)
&+\frac{1}{2}\dv\big(\rho|\u|^2\u\big)+\nabla p\cdot\u \nonumber\\
& =\dv\Psi\cdot
\u+(\nabla\times\H)\times\H\cdot\u
+\nabla\times(\u\times\H)\cdot\H.
\end{align}
Using  the identities
\begin{gather}
 \dv(\H\times(\nabla\times\H))  =|\nabla\times\H|^2-\nabla\times(\nabla\times\H)\cdot\H,\label{naeo}\\
\dv((\u\times\H)\times\H)
=(\nabla\times\H)\times\H\cdot\u+\nabla\times(\u\times\H)\cdot\H,
\label{nae}
\end{gather}
and subtracting \eqref{naaz} from \eqref{nad}, we thus obtain the
internal energy equation
\begin{equation}\label{nagg}
\partial_t (\rho e)+\dv(\rho\u e)+(\dv\u)p=\Psi:\nabla\u+\kappa \Delta \theta,
\end{equation}
where $\Psi:\nabla\u$ denotes the scalar product of two matrices:
\begin{equation*}
\Psi:\nabla\u=\sum^3_{i,j=1}\frac{\mu}{2}\left(\frac{\partial
u^i}{\partial x_j} +\frac{\partial u^j}{\partial
x_i}\right)^2+\lambda|\dv\u|^2=
2\mu|\mathbb{D}(\u)|^2+\lambda(\mbox{tr}\mathbb{D}(\u))^2.
\end{equation*}

Using the Gibbs relation
\begin{equation}\label{gibbs}
\theta \mathrm{d}S=\mathrm{d}e +
p\,\mathrm{d}\left(\frac{1}{\rho}\right),
\end{equation}
we can further replace the equation \eqref{nagg} by
\begin{equation}\label{naf}
\partial_t(\rho S)+\dv(\rho  S\u)=\Psi:\nabla\u+\kappa \Delta \theta,
\end{equation}
where $S$ denotes the entropy.

In the present paper, we assume that $\kappa=0$ in \eqref{naf}. Now,
as in \cite{MS01}, we reconsider the equations of state as functions
of $S$ and $p$, i.e., $\rho=R(S,p)$ and $\theta=\Theta(S,p)$ for
some positive smooth functions $R$ and $\Theta$ defined for all $S$
and $p>0$, and satisfying $\partial R/\partial p >0$. For instance,
we have $\rho=p^{1/\gamma}e^{-S/\gamma}$ for ideal fluids. Then, by
utilizing \eqref{naa} together with the constraint $\dv {\H}=0$, the
system \eqref{naa}, \eqref{nab}, \eqref{nad} and \eqref{naf} can be
rewritten as
\begin{align}
    & A(S,p)(\partial_t p+(\u\cdot \nabla) p)+\dv \u=0,\label{nag}\\
& R(S,p)(\partial_t \u+(\u\cdot \nabla) \u)+\nabla p = (\nabla \times \H)\times \H+\dv\Psi, \label{nah}\\
&  \partial_t {\H} -\cu(\u\times\H)=0, \quad \dv \H=0, \label{nai}\\
 & R(S,p)\Theta(S,p)(\partial_tS+(\u\cdot \nabla) S)=\Psi:\nabla\u,\label{naj}
\end{align}
where
\begin{align}\label{asp}
 A(S,p)=\frac{1}{R(S,p)}\frac{\partial R(S,p)}{\partial p}.
\end{align}

Considering the physical explanation of the incompressible limit, we
introduce the dimensionless parameter $\epsilon$, the Mach number,
and make the following changes of variables:
\begin{gather*}
    p (x, t)=p^\epsilon (x,\epsilon t), \quad S (x, t)=S^\epsilon (x,\epsilon t), \\
    {\u} (x,t)=\epsilon \u^\epsilon(x,\epsilon t), \;\;\;
   {\H} (x,t)=\epsilon \H^\epsilon(x,\epsilon t),
\end{gather*} and
\begin{gather*}
\mu=\epsilon\,\mu^\epsilon,\;\;\;\lambda=\epsilon\,\lambda^\epsilon.
\end{gather*}
As the analysis in \cite{MS01}, we use the
transformation $p^\epsilon (x, \epsilon t)=\underline{p} e^{\epsilon
q^\epsilon(x,\epsilon t)}$ for some positive constant $\underline{p}$.
Under these changes of variables, the system (\ref{nag})--(\ref{naj}) becomes
\begin{align}
    & a^\epsilon(S^\epsilon,\epsilon q^\epsilon)(\partial_t q^\epsilon+(\u^\epsilon\cdot \nabla) q^\epsilon)
    +\frac{1}{\epsilon}\dv \u^\epsilon=0,\label{nak}\\
& r^\epsilon(S^\epsilon,\epsilon q^\epsilon)(\partial_t
\u^\epsilon+(\u^\epsilon\cdot \nabla)
\u^\epsilon)+\frac{1}{\epsilon}\nabla q^\epsilon
 =  ( \cu{\H^\epsilon}) \times {\H^\epsilon}+\dv \Psi^\epsilon,  \label{nal}\\
&  \partial_t {\H}^\epsilon -\cu(\u^\epsilon\times\H^\epsilon)=0,
\quad \dv \H^\epsilon=0, \label{nam}\\
 &b^\epsilon(S^\epsilon,\epsilon q^\epsilon)(\partial_tS^\epsilon+(\u^\epsilon\cdot \nabla)S^\epsilon)
 =\epsilon^2\Psi^\epsilon
 :\nabla\u^\epsilon,\label{nan}
\end{align}
where we have used the abbreviations $\Psi^\epsilon=2\mu^\epsilon
\mathbb{D}(\u^\epsilon)+\lambda^\epsilon\dv\u^\epsilon
\;\mathbf{I}_d$ and
\begin{align}
 a^\epsilon(S^\epsilon,\epsilon
q^\epsilon)& :=  A(S^\epsilon, \underline{p}e^{\epsilon
q^\epsilon})\underline{p}e^{\epsilon q^\epsilon}
=\frac{\underline{p}e^{\epsilon
q^\epsilon}}{R(S^\epsilon,\underline{p}e^{\epsilon
 q^\epsilon})}\cdot
 \frac{\partial R(S^\epsilon,s)}{\partial s}\Big|_{s=\underline{p}e^{\epsilon q^\epsilon}},\label{nann}
\\
  r^\epsilon(S^\epsilon,\epsilon q^\epsilon)& :=  \frac{R(S^\epsilon,\underline{p}e^{\epsilon
 q^\epsilon})}{\underline{p}e^{\epsilon q^\epsilon}}, \quad
 b^\epsilon(S^\epsilon,\epsilon q^\epsilon)
 :=  R(S^\epsilon,\epsilon q^\epsilon)\Theta(S^\epsilon,\epsilon q^\epsilon).\label{nano}
 \end{align}

Formally, we obtain from \eqref{nak} and \eqref{nal} that $\nabla
q^\epsilon \rightarrow 0$ and $\dv \u^\epsilon=0$ as $\epsilon
\rightarrow 0$. Applying the operator \emph{curl} to \eqref{nal},
using the fact that $\cu \nabla =0$,  and letting
$\epsilon\rightarrow 0$ and $\mu^\epsilon\rightarrow \mu>0$, we
obtain
\begin{align*}
\cu\big( r(\bar{S},0)(\partial_t \v+\v\cdot \nabla \v)
 -(\cu\bar{\H}) \times \bar{\H} -\mu\Delta \v\big)=0,
\end{align*}
where we have assumed that
$(S^\epsilon,q^\epsilon,\u^\epsilon,\H^\epsilon)$ and
$r^\epsilon(S^\epsilon,\epsilon q^\epsilon)$ converge to
$(\bar{S},0,\v,\bar{\H})$ and $r(\bar S,0)$ in some sense,
respectively. Finally, applying the identity
\begin{align}\label{naff}
   \cu(  { \u}\times {\H})  =
   {\u} (\dv  {\H})  -   {\H}  (\dv {\u})
 + ( {\H}\cdot \nabla) {\u} - ( {\u}\cdot \nabla) {\H},
 \end{align}  we expect to get the following incompressible non-isentropic MHD
 equations
\begin{align}
&   r(\bar{S},0)(\partial_t \v+(\v\cdot \nabla) \v)
  -(\cu\bar{\H}) \times \bar{\H} +\nabla \pi =\mu\Delta \v, \label{nao} \\
&  \partial_t \bar{\H}  + ( {\v}  \cdot \nabla) \bar{\H}
   - ( \bar{\H} \cdot \nabla) {\v} =0, \label{nap} \\
 &\partial_t \bar{S} +(\v \cdot \nabla) \bar{S} =0, \label{naq}\\
& \dv\,\v=0,  \quad \dv \bar{\H} =0  \label{nar}
\end{align}
for some function $\pi$.

The aim of this paper is to establish the  above limit process
rigorously in the whole space $\mathbb{R}^d$.

Before stating our main results, we review the previous related
works. We begin with the results for the Euler and Navier-Stokes
equations. For well-prepared initial data, Schochet \cite{S86}
obtained the convergence of the compressible non-isentropic Euler
equations to the incompressible non-isentropic Euler equations in a
bounded domain for local smooth solutions. For general initial data,
M\'{e}tivier and Schochet \cite{MS01} proved rigorously the
incompressible limit   of the compressible non-isentropic Euler
equations  in the whole space $\R^d$. There are two key points in
the article \cite{MS01}. First, they obtained the uniform estimates in
Sobolev norms for the acoustic component of the solutions, which are
propagated by a wave equation with unknown variable coefficients.
Second, they proved that the local energy of the acoustic wave
decays to zero in the whole space case. This approach was extended
to the non-isentropic Euler equations in the exterior domain and the
full Navier-Stokes equations in the whole space by Alazard in
\cite{A05} and \cite{A06}, respectively, and to the dispersive
Navier-Stokes equations by Levermore, Sun and Trivisa \cite{LST}.
For the spatially periodic case, M\'{e}tivier and Schochet \cite{MS03} showed
the incompressible limit of the
 one-dimensional non-isentropic Euler equations with general data.
Compared to the non-isentropic case, the treatment of
the propagation of oscillations in the isentropic case is simpler
and there are many works on this topic. For example, see Ukai
\cite{U86}, Asano \cite{As87}, Desjardins and Grenier \cite{DG99} in
the whole space; Isozaki \cite{I87,I89} on the exterior domain;
Iguchi \cite{Ig97} on the half space; Schochet \cite{S94} and
Gallagher \cite{Ga01} in a periodic domain; and Lions and Masmoudi
\cite{LM98}, and Desjardins, et al. \cite{DGLM} in a bounded domain.
Recently, Jiang and Ou \cite{JO} investigated the incompressible
limit of the non-isentropic Navier-Stokes equations with zero heat
conductivity and well-prepared initial data in three-dimensional
bounded domains. The justification of the incompressible limit of the
non-isentropic Euler or Navier-Stokes equations with general initial
data in a bounded domain or a multi-dimensional periodic domain is
still open. The interested reader can refer to \cite{BDGL} for
formal computations on the case of viscous polytropic gases and
\cite{MS03,BDG} for some analysis on the non-isentropic Euler equations
 in a multi-dimensional periodic domain.
For more results on the incompressible limit of the Euler and Navier-Stokes equations,
please see the monograph \cite{FN} and the survey articles \cite{Da05,M07,S07}.

For the isentropic compressible MHD equations, the justification of
the low Mach limit was given in several aspects. In \cite{KM},
Klainerman and Majda first studied the incompressible limit of the
isentropic compressible ideal MHD equations in the spatially periodic case with
well-prepared initial data. Recently,
 the incompressible limit of  the isentropic viscous (including both viscosity
 and magnetic diffusivity) of compressible MHD equations with
 general data was studied in \cite{HW3,JJL1,JJL2}. In
\cite{HW3}, Hu and Wang obtained the convergence of weak solutions
of the compressible viscous MHD equations in bounded,
spatially periodic domains and the whole space, respectively. In \cite{JJL1},
the authors employed the modulated energy method to verify the
limit of weak solutions of the compressible MHD equations in the
torus to the strong solution of the incompressible viscous or
partial viscous MHD equations (the shear viscosity coefficient is
zero but the magnetic diffusion coefficient is a positive constant).
In \cite{JJL2}, the authors obtained the convergence of weak
solutions of the viscous compressible MHD equations to the strong
solution of the ideal incompressible MHD equations in the whole
space by using the dispersion property of the wave equation if both
the shear viscosity and the magnetic diffusion coefficients go to
zero.

For the full compressible MHD equations, the incompressible   limit
  in the framework of the so-called variational
solutions was studied in \cite{K10,KT,NRT}. Recently, the authors \cite{JJL3}
justified rigourously the low Mach number limit of classical solutions to the
  ideal or full compressible non-isentropic MHD equations with small entropy or
  temperature variations. When the heat conductivity and large temperature variations are present,
 the low Mach number limit for the full compressible non-isentropic MHD equations
 justified in~\cite{JJLX}. We emphasize here that the arguments in \cite{JJLX}
 are different from the present paper and
  depend essentially on positivity of fluid viscosity and magnetic diffusivity coefficients.

As aforementioned, in this paper we want to establish rigorously the
limit as $\epsilon\to 0$ to the system \eqref{nak}--\eqref{nan} for
$\mu_\epsilon\to\mu>0$.
 In this case, the magnetic equation is purely hyperbolic due to the lack of magnetic diffusivity. The first-order derivatives of $\H^\epsilon$ in the momentum equation and magnetic equation cannot be controlled.
 It is very hard to study the system  \eqref{nak}--\eqref{nan}. To the author's
knowledge, there is very few mathematical analysis  on the system  \eqref{nak}--\eqref{nan} with fixed or unfixed $\epsilon$,
even for the isentropic case. Our main idea is trying to make full use of the fluid viscosities to control the higher order derivatives of the magnetic field.

Now, we supplement the system \eqref{nak}--\eqref{nan} with initial conditions
\begin{align}
(S^\epsilon,q^\epsilon,\u^\epsilon,\H^\epsilon)|_{t=0}
=(S^\epsilon_0,q^\epsilon_0,\u^\epsilon_0,\H^\epsilon_0)
\label{nas}
\end{align}
and state the main results as follows.
\begin{thm}\label{mth}
 Let $s> d/2+2$ be an integer. Assume that $\mu^\epsilon \rightarrow \mu>0$ and
 $\lambda^\epsilon\rightarrow \lambda$ as $\epsilon\to 0$. Suppose that the initial data
 $(S^\epsilon_0,q^\epsilon_0,\u^\epsilon_0,\H^\epsilon_0)$
satisfy
\begin{align}\label{nat}
\|(S^\epsilon_0,q^\epsilon_0,\u^\epsilon_0,\H^\epsilon_0)\|_{H^s(\R^d)}\leq
M_0. \end{align} Then there exists a $T>0$ such that for any
$\epsilon\in (0,1]$, the Cauchy problem \eqref{nak}--\eqref{nan}, \eqref{nas}
has a unique solution $(S^\epsilon,q^\epsilon,\u^\epsilon,\H^\epsilon)\in
C^0([0,T],H^s(\R^d))$, and there exists a positive constant $N$,
depending only on $T$ and $M_0$, such that
\begin{align}
 \|\big(S^\epsilon,q^\epsilon,\u^\epsilon,\H^\epsilon\big)(t)\|_{H^s(\R^d)}\leq
 N, \quad \forall\, t\in [0,T].
\label{nav}
\end{align}
Furthermore, if there exist positive constants $\underline S$, $N_0$
and $\delta$ such that $S^\epsilon_0(x)$ satisfies
\begin{equation}\label{naw}
|S^\epsilon_0(x)-\underline{S}\,\, |\leq   {N}_0 |x|^{-1-\delta},
\quad |\nabla S^\epsilon_0(x)|\leq  N_0 |x|^{-2-\delta},
\end{equation}
then  the sequence of solutions
$(S^\epsilon,q^\epsilon,\u^\epsilon,\H^\epsilon)$ converges weakly in
  $L^\infty(0,T; H^s(\R^d))$ and strongly in
$L^2(0,T;H^{s'}_{\mathrm{loc}}(\R^d))$ for all $s'<s$ to a limit
$(\bar S,0,\v,\bar\H)$, where \linebreak  $(\bar S,\v,\bar{\H})$
is the unique solution in $C([0,T],H^s(\R^d))$ of \eqref{nao}--\eqref{nar} with
initial data $(\bar S,\v,\bar{\H})|_{t=0} =( S_0,\w_0,\H_0)$,
where $\w_0\in H^s(\R^d)$ is determined by
\begin{equation}\label{nax}
   \dv\,\w_0=0,  \,\; \cu(r(S_0,0)\w_0)=\cu(r(S_0,0)\v_0),
\;\, r(S_0,0):= \lim_{\epsilon \rightarrow
0}r^\epsilon(S^\epsilon_0,0).
\end{equation}
The function $\pi\in C([0,T]\times \R^d)$ satisfies $\nabla\pi\in
C([0,T],H^{s-1}(\R^d))$.
\end{thm}

We briefly describe the strategy of the proof.
 The proof of Theorems \ref{mth} includes two main steps:
the uniform estimates of the solutions, and the convergence from the
original scaling equations to the limiting ones. Once we have
established the uniform estimates \eqref{nav} of the solutions in
Theorems \ref{mth}, the convergence of solutions is easily proved by
using the local energy decay theorem for fast waves in the whole
space, which is shown by M\'{e}tivier and Schochet in \cite{MS01}.
Thus, the main task in the present paper is to obtain the uniform
estimates \eqref{nav}. For this purpose, we shall modify the
approach developed in \cite{MS01}.
 In fact, due to  the strong coupling of hydrodynamic motion and
magnetic field, and the lack of magnetic diffusivity, new
difficulties arise in obtaining the uniform estimates for the solutions
to \eqref{nak}--\eqref{nan}, \eqref{nas}. First of all, when we perform the operator
$(\{E^\epsilon\}^{-1}L(\partial_x))^\sigma$ to the continuity and momentum
equations, or the operator $\emph{curl}$ to the momentum equations,
 one order more spatial derivatives arise for the magnetic field, and this prevents us
 from closing the energy estimates. Second, since the
coefficients of the acoustic wave equations depend on the entropy,
we could not get the estimates of $\| \u^\epsilon\|_{L^2(0,T;H^{s+1})}$
directly from the system. The ideas to overcome these difficulties here are
the following: We transfer one spatial derivative from the
magnetic field to the velocity with the help of the special coupled
way between magnetic field and fluid velocity. Then, to control
$\|\u^\epsilon\|_{L^2(0,T;H^{s+1})}$, we employ a kind of Helmholtz decomposition of the velocity. Third, we make full use the
special structure of the magnetic field equation and the estimates
on $\u$ to control $\| \H^\epsilon\|_{L^\infty(0,T;H^{s})}$.

We point out that our arguments in this paper can be modified
slightly to the case of the  the compressible non-isentropic MHD
equations with infinite Reynolds number. We shall give a brief
discussion in Section 5.

This paper is arranged as follows. In Section 2, we give notations, recall basic facts, and
 present commutators estimates. In Section 3 we establish the uniform boundeness of
   the solutions and prove the existence part of Theorem \ref{mth}.
In Section 4, we use the decay of the local energy to the acoustic wave equations
  to prove the convergent part of Theorem \ref{mth}.
   In the last section, we consider the incompressible limit to the compressible
   non-isentropic MHD equations with infinite Reynolds number.

\section{Preliminary}

We give notations and recall basic facts which will be used frequently in the proofs.

(1) We denote  $\langle \cdot,\cdot\rangle$ the standard inner
product in $L^2(\R^d)$ with $\langle f,f\rangle=\|f\|^2$ and
 $H^k$ the usual Sobolev space $W^{k,2}$ with norm $\|\cdot\|_{k}$, in particular,
$\|\cdot\|_0=\|\cdot\|$. The notation $\|(A_1,  \dots, A_k)\|$ means
the summation of $\|A_i\|$ ($i=1,\cdots,k$), and it also applies to other norms.
For a multi-index $\alpha = (\alpha_1,\dots,\alpha_d)$, we denote
$D^\alpha =\partial^{\alpha_1}_{x_1}\dots\partial^{\alpha_d}_{x_d}$ and
$|\alpha|=|\alpha_1|+\cdots+|\alpha_d|$. We also omit the spatial
domain $\R^d$ in integrals for convenience. We use the symbols $K$
or $C_0$ to denote the generic positive constants, and $C(\cdot)$
and $\tilde{C}(\cdot)$ to denote  the smooth functions, which may
vary from line to line.

(2) For a scalar function $f$ and vector functions
$\mathbf{a}$,  $\mathbf{b}$ and $\mathbf{c}$, we have the following basic vector identities:
\begin{align}
   \dv (\mathbf{a}\times \mathbf{b})& =\mathbf{b}\cdot \cu \mathbf{a} -\mathbf{a}\cdot \cu \mathbf{b},\label{va}\\
  \nabla(|\mathbf{a}|^2)& =2(\mathbf{a}\cdot \nabla )\mathbf{a}+2\mathbf{a}\times \cu \mathbf{a},\label{vbb}\\
\cu(f \mathbf{a}) & =f\cdot \cu \mathbf{a} -\nabla f\times   \mathbf{a},\label{vbc}\\
    \cu(\mathbf{a}\times \mathbf{b})& =(\mathbf{b}\cdot \nabla) \mathbf{a} - (\mathbf{a}\cdot \nabla) \mathbf{b}
  + \mathbf{a} (\dv \mathbf{b})  -   \mathbf{b} (\dv \mathbf{a}),\label{vb}\\
  \dv \big((\mathbf{a}\times \mathbf{b})\times \mathbf{c}\big)
 &=\mathbf{c}\cdot \cu (\mathbf{a}\times \mathbf{b}) -(\mathbf{a}\times \mathbf{b})\cdot \cu \mathbf{c}.\label{vc}
 \end{align}

(3) We have the following well-known nonlinear estimates \cite{Ho97}.

\hskip 4mm (i)\ \ Let  $\alpha=(\alpha_1,\alpha_2,\alpha_d)$ be a
multi-index such that $|\alpha|=k$. Then, for all
$\sigma\geq 0$, and $f,g\in H^{k+\sigma}(\R^d)$, there exists a
generic constant $C_0$ such that
\begin{align}\label{comc}
    \|[f,\partial^\alpha ]g\|_{H^{\sigma}}\leq & C_0(\|f\|_{W^{1,\infty}}\| g\|_{H^{\sigma+k-1}}
    +\| f\|_{H^{\sigma+k}}\|g\|_{L^{\infty}}).
\end{align}

\hskip 4mm (ii)\ \ For integers $ k\geq 0$, $l\geq 0$, $k+l\leq
\sigma$ and $\sigma > d/2$, the product maps continuously
$H^{\sigma -k}(\R^d)\times H^{\sigma-l}(\R^d)$  to
$H^{\sigma-k-l}(\R^d)$ and
\begin{align}\label{nay}
 \|uv\|_{\sigma-k-l}\leq K\|u\|_{\sigma-k}\|v\|_{\sigma-l}.
\end{align}

\hskip 4mm (iii)\ \  Let $\sigma>d/2$ be an integer.  Assume that
$F(u)$ is a smooth function such that $F(0)=0$ and $u\in
H^{\sigma}(\R^d)$, then $F(u)\in H^{\sigma}(\R^d)$ and its norm  is
bounded by
\begin{align}\label{naz}
\|F(u)\|_\sigma\leq C(\|u\|_\sigma)\|u\|_\sigma,
\end{align}
 where $C(\cdot)$ is independent of $u$ and maps $[0,\infty)$ into
 $[0,\infty)$.

\section{Uniform estimates}

In this section and the first part of the next section we assume
that $\mu^\epsilon \equiv \mu>0$ and $\lambda^\epsilon \equiv
\lambda$
 for simplicity of the presentation.
The general case  can be treated by a slight modification in the
arguments presented here.

In view of  \cite{MS01} and the classical local existence results obtained by
Vol'pert and Khudiaev \cite{VK} for hyperbolic-parabolic systems,
 the key point in the proof of the existence part of Theorem \ref{mth} is to
 establish  the uniform estimate \eqref{nav}, which can be
deduced from the following \emph{a priori} estimate.
\begin{thm}\label{prop}
For any $\epsilon>0$, let
$(S^\epsilon,q^\epsilon,\u^\epsilon,\H^\epsilon)\in
C([0,T],H^s(\R^d))$ be the solution to \eqref{nak}--\eqref{nan}.
Then there exists an increasing $C(\cdot)$ from $[0,\infty)$ to
$[0,\infty)$, such that
\begin{align}\label{nbc}
\mathcal{M}_\epsilon(T) \leq C_0 + (T +\epsilon )C(\mathcal{M}_\epsilon(T)),
 \end{align}
 where
 \begin{align}
\mathcal{M}_\epsilon(T) :=\,&  { \mathcal{N}_\epsilon(T)}^2
+\int_0^T\|\u^\epsilon\|_{s+1}^2d\tau,\label{nbb}
\end{align}with
\begin{align}
  \mathcal{N}_\epsilon(T) :=\,&  \sup_{ t\in  [0,T
]}\|(S^\epsilon,q^\epsilon,\u^\epsilon,\H^\epsilon)(t)\|_{s}.\label{nbbbb}
\end{align}
\end{thm}

The remainder of this section is devoted to establishing \eqref{nbc}. In the calculations
that follow, we always suppose that the assumptions in Theorem
\ref{mth} hold. We consider a solution
$(S^\epsilon,q^\epsilon,\u^\epsilon,\H^\epsilon)$ to the problem
\eqref{nak}--\eqref{nan}, \eqref{nas} on $C([0,T],H^s(\R^d))$ with
initial data satisfying \eqref{nat}.

The main idea for proving the uniform estimate \eqref{nbc} is
motivated by the work \cite{MS01} where the operator
$(\{E^\epsilon\}^{-1}L(\partial_x))^m$ is introduced to control the acoustic
components of velocity for the Euler equations. When the strong coupling of the
fluid and magnetic filed is present, however, the arguments in \cite{MS01} cannot be
directly applied to get a uniform estimate of the acoustic parts due to lack of magnetic
diffusion in the magnetic equation. Instead, here we transfer one
order spatial derivative from $\H^\epsilon$ to $\u^\epsilon$, and then employ the fluid viscosity
to control higher derivatives. We remark that the reason that
these techniques work is due to the special structure of coupling between the fluid
and magnetic fields.

We begin with the estimate on the entropy $S^\epsilon$.

\begin{lem}  \label{NLa}
There exist a constant $C_0>0$   and a function $C(\cdot)$,
independent of $\epsilon$,  such that for all $ t\in (0,T]$,
\begin{align}\label{nbe}
\|S^\epsilon(t)\|^2_{s}\leq  C_0 + tC(\mathcal{M}_\epsilon(T))+\epsilon^2
C(\mathcal{M}_\epsilon(T)).
\end{align}
\end{lem}
\begin{proof}
For the multi-index $\alpha$ satisfying $|\alpha| \leq s-1$, denote
$f_\alpha=\partial_x^\alpha S^\epsilon$. In view of the positivity
of $b^\epsilon(S^\epsilon,\epsilon q^\epsilon)$, we deduce from
\eqref{nan} that,
\begin{align}\label{nbaa}
\partial_t f_\alpha +(\u^\epsilon\cdot \nabla)f_\alpha =g_\alpha+\epsilon^2
h_\alpha,
\end{align}
where
\begin{align*}
g_\alpha=-[\partial_x^\alpha, \u^\epsilon]\cdot \nabla
S^\epsilon,\;\;h_\alpha=\epsilon^2
\partial^\alpha_x\left(\frac{\Psi^\epsilon
 :\nabla\u^\epsilon}{b^\epsilon(S^\epsilon,\epsilon
  q^\epsilon)}\right).
\end{align*}

The commutator inequality \eqref{comc} and Sobolev embedding theorem  imply that
   $\|g_\alpha\|\leq C(M_\epsilon(T))$. On the other hand, from the Sobolev embedding theorem
   and the Moser-type inequality \cite{KM} we get
\begin{align*}
   \|h_\alpha\|& \leq K\Big( \|(\Psi^\epsilon:\nabla\u^\epsilon)\|_{L^\infty}
   \Big\|D^{s}\Big(\frac{1}{b^\epsilon}\Big)\Big\|
   +\|D^{s}(\Psi^\epsilon
 :\nabla\u^\epsilon)\|\Big\|\frac{1}{b^\epsilon}\Big\|_{L^\infty}\Big)\\
   &\leq C(\mathcal{N}_\epsilon(T))+C(\mathcal{N}_\epsilon(T))\|\u^\epsilon\|_{s+1}.
\end{align*}
Multiplying \eqref{nbaa} by $f_\alpha$ and integrating over
$[0,t]\times \R^d$ with $t\leq T$, we obtain
\begin{align*}
 \|f_\alpha(t)\|^2 \leq & \|f_\alpha(0)\|^2 +\|\partial_x\u^\epsilon\|_{L^\infty((0,t)\times\R^d)}
  \int^t_0 \|f_\alpha(\tau)\|^2 d\tau\\
                       & + 2\int^t_0 \|g_\alpha(\tau)\|\,\|f_\alpha(\tau)\| d\tau +
                       2 \epsilon^2  \int^t_0 \|h_\alpha(\tau)\|\, \|f_\alpha(\tau)\| d\tau \\
             \leq & C_0+tC(\mathcal{M}_\epsilon(T))+\epsilon^2 C(\mathcal{M}_\epsilon(T)),
\end{align*}
 where we have used Young's inequality  and the embedding
$H^\sigma\hookrightarrow L^\infty$ for $\sigma> d/2$. The conclusion
then follows by
  adding up these estimates for all $|\alpha|\leq s-1$.
\end{proof}

The following $L^2$-bound of
$(q^\epsilon,\u^\epsilon,\H^\epsilon)$ can be obtained directly using the energy method due
to the skew-symmetry of the singular term in the system and the special structure of
coupling between the magnetic field and fluid velocity. This $L^2$-bound
is very important in our arguments, since the induction
analysis will be used to get the desired Sobolev estimates.

\begin{lem} \label{NLb}
  There exist   constants $C_0>0$ and $0<\xi_1<\mu$,  and a function $C(\cdot)$ independent of
  $\epsilon$, such that for all $ t\in [0,T]$,
\begin{align}\label{nbf}
\|(q^\epsilon,\u^\epsilon,\H^\epsilon)(t)\|^2 +\xi_1 \int^t_0
\|\nabla \u^\epsilon(\tau)\|^2d\tau \leq  C_0 + tC(\mathcal{M}_\epsilon(T)).
\end{align}
\end{lem}
\begin{proof} Multiplying \eqref{nak} by $q^\epsilon$, \eqref{nal} by
$\u^\epsilon$, and \eqref{nam} by $\H^\epsilon$, respectively,
integrating over $\R^d$, and adding the resulting equations together, we obtain
\begin{align}
    &\langle a^\epsilon \partial_t q^\epsilon, q^\epsilon\rangle
+  \langle r^\epsilon \partial_t \u^\epsilon, \u^\epsilon\rangle
   +\langle  \partial_t \H^\epsilon, \H^\epsilon\rangle
   +\mu \|\nabla \u^\epsilon\|^2+
   (\mu+\lambda)\|\dv \u^\epsilon\|^2\nonumber\\
      &\qquad +\langle a^\epsilon (\u^\epsilon\cdot \nabla) q^\epsilon, q^\epsilon\rangle
   + \langle r^\epsilon (\u^\epsilon\cdot \nabla) \u^\epsilon,
   \u^\epsilon\rangle\nonumber\\
   & =\int[( \cu{\H^\epsilon}) \times
 {\H^\epsilon}]\cdot\u^\epsilon\;dx+\int\cu(\u^\epsilon\times\H^\epsilon)\cdot
 \H^\epsilon\;dx.\label{first1}
  \end{align}
Here the singular terms involving $1/\epsilon$  are canceled. Using
the identity \eqref{nae} and integrating by parts, we immediately
obtain that
\begin{align*}
\int[( \cu{\H^\epsilon}) \times
 {\H^\epsilon}]\cdot\u^\epsilon\;dx+\int\cu(\u^\epsilon\times\H^\epsilon)\cdot
 \H^\epsilon\;dx=0.
\end{align*}

In view of the positivity and smoothness of $a^\epsilon(S^\epsilon,
\epsilon q^\epsilon)$ and $r^\epsilon(S^\epsilon, \epsilon q^\epsilon)$,
we get directly from \eqref{nak}, \eqref{nan} and \eqref{naz} that
\begin{align}\label{nbba}
  \| \partial_t S^\epsilon\|_{s-1}\leq C(\mathcal{N}_\epsilon(T)), \quad
  \| \epsilon\partial_t q^\epsilon\|_{s-1}\leq C(\mathcal{N}_\epsilon(T)),
\end{align}
while by the Sobolev embedding theorem, we find that
\begin{align}\label{nbbb}
\|(\partial_{t}a^\epsilon, \partial_{t}r^\epsilon)\|_{L^\infty}\leq
 \|(\partial_{t}a^\epsilon,\partial_{t}r^\epsilon)\|_{s-2}\leq
 C(\mathcal{N}_\epsilon(T)).
\end{align}
By the definition of $\mathcal{N}_\epsilon(T)$ and the Sobolev embedding theorem,
it is easy to see that
\begin{align*}
\|(\nabla a^\epsilon, \nabla r^\epsilon)\|_{L^\infty}\leq C(\mathcal{N}_\epsilon(T)).
\end{align*}

Since $\mu>0, 2\mu+d\lambda>0$, there exists a positive constant
$\kappa_1$ such that
\begin{align*}
\mu \|\nabla \u^\epsilon\|^2+
   (\mu+\lambda)\|\dv \u^\epsilon\|^2\geq \kappa_1 \|\nabla\u^\epsilon\|^2.
\end{align*}
Thus, from \eqref{first1} we get that
\begin{align}\label{nbg}
 \langle a^\epsilon   q^\epsilon, q^\epsilon\rangle
   +&  \langle r^\epsilon  \u^\epsilon, \u^\epsilon\rangle
    +\langle   \H^\epsilon, \H^\epsilon\rangle+
    \kappa_1 \int^t_0 \|\nabla \u^\epsilon(\tau)\|^2d\tau\nonumber\\
  \leq & \big\{ \langle a^\epsilon   q^\epsilon, q^\epsilon\rangle
   +  \langle r^\epsilon  \u^\epsilon, \u^\epsilon\rangle+\langle
   \H^\epsilon, \H^\epsilon\rangle\big\}\big|_{t=0}      \nonumber\\
  &  + C(\mathcal{M}_\epsilon(T))\int^t_0  \big\{|q^\epsilon(\tau)|^2   +|\u^\epsilon(\tau)|^2
 +|\H^\epsilon(\tau)|^2\big\}d\tau.
   \end{align}
Moreover, we have
\begin{align}\label{aaaaa}
\|q^\epsilon\|^2+\|\u^\epsilon\|^2&\leq
\|(a^\epsilon)^{-1}\|_{L^\infty}\langle a^\epsilon q^\epsilon,
q^\epsilon\rangle
   + \|(r^\epsilon)^{-1}\|_{L^\infty}\langle r^\epsilon  \u^\epsilon,
   \u^\epsilon\rangle\nonumber\\
&\leq C_0(\langle a^\epsilon q^\epsilon,
q^\epsilon\rangle+\langle r^\epsilon\u^\epsilon, \u^\epsilon\rangle),
\end{align}
since $a^\epsilon$ and $r^\epsilon$ are uniformly bounded away from
zero. Applying Gronwall's Lemma to \eqref{nbg}, we conclude
\begin{align*}
   \|(q^\epsilon, \u^\epsilon,  \H^\epsilon)(t)\|^2\le C_0 \|(q^\epsilon_0, \u^\epsilon_0,
   \H^\epsilon_0)\|^2\exp\{tC(\mathcal{M}_\epsilon(T))\}.
\end{align*}
Therefore, the estimate \eqref{nbf} follows from an elementary
inequality
\begin{align}\label{ele}
e^{Ct}\leq 1+\tilde{C}t, \quad 0 \leq t \leq T_0,
\end{align}
where $T_0$ is some fixed constant.
\end{proof}

Concerning the desired higher order estimates, we cannot directly get them by differentiating
the system as done in \cite{MS01}, since the coefficients $a^\epsilon(S^\epsilon, \epsilon
q^\epsilon), r^\epsilon(S^\epsilon, \epsilon q^\epsilon)$ and
$b^\epsilon(S^\epsilon, \epsilon q^\epsilon)$ contain two scales
$S^\epsilon$ and $\epsilon q^\epsilon$. We shall adapt and modify the
techniques developed in \cite{MS01} to derive the higher order estimates. Set
\begin{gather*}
  E^\epsilon(S^\epsilon,\epsilon q^\epsilon)=\left(\begin{array}{cc}
                   a^\epsilon(S^\epsilon,\epsilon q^\epsilon) & 0 \\
                    0 & r^\epsilon(S^\epsilon,\epsilon q^\epsilon)\mathbf{I}_{d}\\
                   \end{array}\right),\quad
    \U^\epsilon=\left(\begin{array}{c}
                   q^\epsilon \\
                   \u^\epsilon
                    \end{array}
                    \right),
                  \\[1mm]
      L(\partial_x)=\left(\begin{array}{cc}
                    0 & \dv  \\
                    \nabla & 0\\
                    \end{array}\right),
     \end{gather*}
where $\mathbf{I}_{d}$ denotes the  $d\times d$ unit matrix.

 Let $L_{E^\epsilon}(\partial_x )= \{E^\epsilon\}^{-1}L(\partial_x)$ and
$r_0(S^\epsilon) = r^\epsilon(S^\epsilon,0)$. Note that
$r_0(S^\epsilon)$ is smooth, positive, and bounded away from zero
with respect to each $\epsilon$. First, using Lemma \ref{NLa} and employing the same
analysis as in \cite{MS01}, we have

\begin{lem}\label{NLi}
  There exist constants $C_1>0$, $K>0$, and a function $C(\cdot)$, depending
only on $M_0$, such that for all $\sigma \in [1,\dots, s]$ and  $t\in [0,T]$,
\begin{align}\label{nbh}
  \|\U^\epsilon\|_\sigma \leq K\|L(\partial_x) {\U}^\epsilon\|_{\sigma -1}
+\tilde{ C}\big(\|\cu
(r_0\u^\epsilon)\|_{\sigma-1}+\|\U^\epsilon\|_{\sigma-1}\big)
\end{align}
and
\begin{align}\label{nbha}
\|\U^\epsilon\|_\sigma \leq \tilde{
C}\big\{\|\{L_{E^\epsilon}(\partial_x)\}^\sigma  {\U}^\epsilon\|_{0}
+\|\cu(r_0\u^\epsilon)\|_{\sigma-1}+\|\U^\epsilon\|_{\sigma-1}\big\},
\end{align}
where $\tilde{C}:=C_1+tC(\mathcal{M}_\epsilon(T))+\epsilon C(\mathcal{M}_\epsilon(T)) $.
\end{lem}
We remark that the inequalities \eqref{nbh} and \eqref{nbha} are similar to the well known Helmholtz decomposition,  and the estimate on $\|S^\epsilon(t)\|^2_{s}$ in Lemma
\ref{NLa} plays a key role in the proof of Lemma \ref{NLi}.

Our next task is to derive a bound on
$\|\{L_{E^\epsilon}(\partial_x)\}^\sigma  {\U}^\epsilon\|_{0}$ and
$\|\cu(r_0\u^\epsilon)\|_{\sigma-1}$ by induction arguments. Let
$\W_\sigma^\epsilon=\{L_{E^\epsilon}(\partial_x)\}^\sigma(0,\u^\epsilon)^\top$.
We first show the following estimate.

\begin{lem}\label{NLee}
  There exist a sufficiently small constant $\eta_1>0$ and two constants
  $C_0>0$, $0<\xi_2<\mu$, and a function $C(\cdot)$
from $[0,\infty)$ to $[0,\infty)$, independent of $\epsilon$,  such
that for all $\sigma \in [1,\dots, s]$ and  $t\in [0,T]$,
\begin{align}\label{nbeea}
\| \{L_{E^\epsilon}(\partial_x)\}^\sigma  {\U}^\epsilon(t)\|^2
&+\frac{\xi_2}{2}\int_0^t\|\nabla\W_\sigma^\epsilon\|^2(\tau)d\tau \nonumber\\
& \leq C_0
+ tC(\mathcal{N}_\epsilon(T)) +\eta_1\int_0^t\|\u^\epsilon(\tau)\|_{s+1}^2d\tau.
\end{align}
\end{lem}

\begin{proof} Let $\U^\epsilon_\sigma:=\{L_{E^\epsilon}(\partial_x)\}^\sigma \U^\epsilon$,
$\sigma\in \{0,\dots, s\}.$ For simplicity, we set $\mathcal{M}:=\mathcal{M}_\epsilon(T)$,
$\mathcal{N}:=\mathcal{N}_\epsilon(T)$, and $E:= E^\epsilon(S^\epsilon,\epsilon q^\epsilon)$.
 The case $k =0$ is an immediate consequence of Lemma \ref{NLb}.
It is easy to verify that the operator $L_E(\partial_x )$ is
bounded from $H^{k}$ to $H^{k-1}$ for $k \in \{1,\dots, s+1\}$. Note that
the equations \eqref{nak}, \eqref{nal} can be written as
\begin{align}\label{nbeb}
  (\partial_t+\u^\epsilon\cdot \nabla)\U^\epsilon+\frac{1}{\epsilon}E^{-1}L(\partial_x)\U^\epsilon=
  E^{-1}(\mathbf{J}^\epsilon+\mathbf{V}^\epsilon)
\end{align}
with
\begin{gather*}
\mathbf{J}^\epsilon=\left(\begin{array}{c}
                 0 \\
                ( \cu{\H^\epsilon}) \times {\H^\epsilon}
                 \end{array}
                 \right),\quad
\mathbf{V}^\epsilon=\left(\begin{array}{c}
                 0 \\
                \dv \Psi^\epsilon
                 \end{array}
                 \right).
\end{gather*}

 For $ \sigma \geq 1$, we commute the operator $\{L_E\}^\sigma$ with  \eqref{nbeb} and
 multiply the resulting equation by $E$ to infer that
\begin{align}\label{nbeeb}
  E(\partial_t+\u^\epsilon\cdot \nabla)\U^\epsilon_\sigma
  +\frac{1}{\epsilon}L(\partial_x)\U^\epsilon_\sigma=E(\mathbf{f}_\sigma+
  \mathbf{g}_\sigma+\mathbf{h}_\sigma),
\end{align}
where
\begin{align*}
\mathbf{f}_\sigma:=&   [\partial_t+\u^\epsilon\cdot \nabla , \{L_E\}^\sigma] \U^\epsilon,\\
\mathbf{g}_\sigma:=&  \{L_E\}^\sigma (E^{-1} \mathbf{J}^\epsilon),\\
\mathbf{h}_\sigma:=&  \{L_E\}^\sigma (E^{-1} \mathbf{V}^\epsilon).
   \end{align*}

Multiplying \eqref{nbeeb} by $ \U^\epsilon_\sigma$ and integrating
over $[0,t]\times \R^d$ with $t\leq T$, noticing that
 the singular terms cancel out since $L(\partial_x)$ is skew-adjoint, we use
the inequalities \eqref{nbba} and \eqref{nbbb}, and Cauchy-Schwarz's
inequality to deduce that
\begin{align}\label{nbeec}
  \langle E(t)\U_\sigma^\epsilon(t), \U_\sigma^\epsilon(t)\rangle
\leq & \langle E(0)\U_\sigma^\epsilon(0),
\U_\sigma^\epsilon(0)\rangle
         + C(\mathcal{M})\int^t_0 \|\U^\epsilon_\sigma(\tau)\|^2 d\tau \nonumber\\
  & + \int^t_0 \|\mathbf{f}_\sigma(\tau)\|^2 d\tau+
2\int^t_0\int_{\R^d}(E (\mathbf{g}_\sigma+\mathbf{h}_\sigma)\U_\sigma^\epsilon)(\tau) d\tau.
\end{align}
Following the proof process of Lemma 2.4 in \cite{MS01}, we obtain that
\begin{align}\label{nbeef}
 \|\mathbf{f}_k(t)\|\leq C(\mathcal{N}_\epsilon(t)).
\end{align}

Now we estimate the nonlinear term in \eqref{nbeec} involving
$\mathbf{g}_\sigma$. We expand $\mathbf{g}_\sigma$ as follows
\begin{align*}
   \mathbf{g}_\sigma=& \, \sum^d_{i,j=1}\sum_{|\alpha|=\sigma+1}
   \{E^{-1}\}^{k+1}\partial_x^\alpha H_i^\epsilon H^\epsilon_j\nonumber\\
   & + \sum^d_{i,j=1}\sum_{\Lambda_1}  \sum_{\Lambda_2}\{E^{-1}\}^l
   \partial_x^{\beta_1}\{E^{-1}\}\cdots
   \partial_x^{\beta_k}\{E^{-1}\}\partial_x^\gamma H_i^\epsilon \partial_x^\delta H^\epsilon_j \nonumber\\
:=&\, B_1+B_2,
\end{align*}
where
 \begin{gather*}
\Lambda_1   =\{(\beta_1,\cdots,\beta_k,\gamma,\delta)\big|
|\beta_1|+\cdots+|\beta_k|+|\gamma|+|\delta|\leq k+1, 0<|\gamma|
\leq k, |\delta|\leq k\},\\
\Lambda_2   =\{l\big|l=k+1-(|\beta_1|+\cdots+|\beta_k|),
(\beta_1,\cdots,\beta_k,0,0)\in \Lambda_1\}.
\end{gather*}
Since there is no magnetic diffusion in the system, we cannot deal
with directly the terms involving $B_1$. Instead, we
transform one spatial derivative to $\U_\sigma^\epsilon$. Integrating by parts, we have
\begin{align}
\int_{\R^d} (E B_1\U_\sigma^\epsilon)(\tau)dx=&
-\sum^d_{i,j=1}\sum_{|\alpha|=\sigma} \int_{\R^d}
\{E^{-1}\}^{k}\partial_x^\alpha  H_i^\epsilon
\partial_x H^\epsilon_j\U_\sigma^\epsilon d x \nonumber\\
&-\sum^d_{i,j=1}\sum_{|\alpha|=\sigma} \int_{\R^d}
\partial_x\{E^{-1}\}^{k}\partial_x^\alpha  H_i^\epsilon
H^\epsilon_j\U_\sigma^\epsilon dx \nonumber\\
&-\sum^d_{i,j=1}\sum_{|\alpha|=\sigma}\int_{\R^d}
\{E^{-1}\}^{k}\partial_x^\alpha  H_i^\epsilon
 H^\epsilon_j\partial_x(\U_\sigma^\epsilon) dx \nonumber\\
 \leq &C(\mathcal{N})+\eta_1 \|\u^\epsilon(\tau)\|_{s+1}^2
\end{align}
for sufficiently small constant $\eta_1>0$.

 By virtue of Cauchy-Schwarz's and Sobolev's inequalities, and \eqref{nbbb},
a direct computation implies that
\begin{align}
 \int_{\R^d}|(E B_2\U_\sigma^\epsilon)(\tau)|d\tau\leq C(\mathcal{N}).
 \label{nbeeh}
\end{align}

Next, we deal with the term involving the viscosity. Recall that
$L(\partial_x) \U^\epsilon=(\dv \u^\epsilon, \nabla q^\epsilon)$. Denote
\begin{align}
L_1:= \{a^\epsilon\}^{-1}\dv,\;\;\;\;
L_2:=\{r^\epsilon\}^{-1}\nabla. \nonumber
\end{align}
A straightforward computation implies that
\begin{align*} 
  \U^\epsilon_k= \left\{
                   \begin{array}{ll}
                   \left(\begin{array}{c}
                   \{L_1L_2\}^{\frac{k-1}{2}} L_1  \u^\epsilon \\
                   \{L_2 L_1\}^{\frac{k-1}{2}} L_2  q^\epsilon
                    \end{array}
                    \right)  , &  \text{if}\  k  \ \text{is  odd}; \\[-1.0ex]\\
                    \left(\begin{array}{c}
                   \{L_1 L_2\}^{k/2} q^\epsilon \\
                   \{L_2 L_1\}^{k/2} \u^\epsilon
                    \end{array}
                    \right), & \hbox{\text{if}}\  k \  \text{is even}.
                   \end{array}
                 \right.
  \end{align*}
Thus, we induce that
\begin{align*}
\int^t_0\int_{\R^d} (E \mathbf{h}_\sigma \U_\sigma^\epsilon)(\tau)
dxd\tau=\int^t_0\int_{\R^d} E\{L_E\}^\sigma (E^{-1}
\mathbf{V}^\epsilon)\W_\sigma^\epsilon dxd\tau.
\end{align*}
An integration by parts gives
\begin{align*}
& \int^t_0\int_{\R^d} EL^\sigma_E (E^{-1}
\mathbf{V}^\epsilon)\W_\sigma^\epsilon dx
d\tau\\
= &-\int^t_0\int_{\R^d}\mu|\nabla\W_\sigma^\epsilon|^2+
(\mu+\lambda)|\dv\W_\sigma^\epsilon|^2dxd\tau\\
& +\int^t_0\int_{\R^d}E[\mu E^{-1}\Delta+(\mu+\lambda) E^{-1}\nabla\dv,
\{L_E\}^\sigma ](0,\u^\epsilon)^T\W_\sigma^\epsilon dxd\tau ,
\end{align*}
where it is easy to verify that
\begin{align}\label{viscosity}
[E^{-1}\Delta,\{L_E\}^\sigma ]=\sum_{i=0}^{k-1}\{L_E\}^i[E^{-1}\Delta,L_E]\{L_E\}^{\sigma-i-1}.
\end{align}

Noting that $L_E(\partial_x )= {E}^{-1}L(\partial_x)$, we find that
\begin{align*}
[E^{-1}\Delta,L_E]=-E^{-1}\Delta
E^{-1}L(\partial_x)+\sum_{i,j=1}^dB_{ij}\partial_{x_{ij}},
\end{align*}
where $B_{ij}$ $(i,j=1,\cdots, d$) are the sums of bilinear functions
of $E^{-1}$ and $\partial_x \{E^{-1}\}$, and Sobolev's inequalities imply that
\begin{align*}
\|B_{ij}\|_{s-1}\leq C(\mathcal{N}).
\end{align*}
Thus, we integrate by parts to infer that
\begin{align*}
&\mu\int^t_0\int_{\R^d}E[ E^{-1}\Delta,
\{L_E\}^\sigma](0,\u^\epsilon)^\top\W_\sigma^\epsilon dxd\tau \\
 \leq & \frac{\mu}{2}\int_0^t\|\nabla
\W_\sigma^\epsilon\|^2d\tau+C(\mathcal{N})\int_0^t\|\W_\sigma^\epsilon(\tau)\|^2d\tau\\
&  +\int_0^t\|\tilde{H}_2^\sigma(\tau)\|^2+\|\tilde{H}_1^\sigma(\tau)\|^2 d\tau.
\end{align*}
Here $\tilde{H}_1^\sigma$ is a finite sum of terms of the form
$$ (\partial^{\alpha_1}_x e_1)\cdots (\partial^{\alpha_l}_x e_l) \,
(\partial^{\beta}_x w)\,(\partial^{\gamma}_x u^\epsilon_m)
$$
with $|\alpha_1|+\cdots +|\alpha_l|+|\beta|+|\gamma|\leq \sigma\leq
s$, $|\gamma|> 0$, and thus $|\beta|\leq k-1\leq s-1$, where
$(e_1, \dots, e_l)$, $w$ and $u^\epsilon_m$ denote the
coefficients of $E^{-1}$, $C_j$ and $\u^\epsilon$ respectively,
with $C_j$ taking a form similar to that of $B_{ij}$.  $\tilde{H}_2^\sigma$ is a
finite sum of terms of the form
$$
(\partial^{\alpha_1}_x e_1)\cdots (\partial^{\alpha_l}_x e_l) \,
(\partial^{\beta}_x w)\,(\partial^{\gamma}_x u^\epsilon_m)
$$
with $|\alpha_1|+\cdots +|\alpha_l|+|\beta|+|\gamma|\leq
\sigma+1\leq s+1$, $|\gamma|> 1$, and thus $|\beta|\leq \sigma-1\leq s-1$,
where $(e_1, \dots, e_l)$, $w$ and $u^\epsilon_m$
denote the coefficients of $E^{-1}$, $B_{ij}$ and $\u^\epsilon$ respectively.
Hence, we have
\begin{align*}
\|\tilde{H}_1^\sigma\|^2+\|\tilde{H}_2^\sigma\|^2\leq C(\mathcal{M}).
\end{align*}

Similarly, we can show that
\begin{align*}
&(\mu+\lambda)\int^t_0\int_{\R^d}E[ E^{-1}\nabla\dv,
\{L_E\}^\sigma](0,\u^\epsilon)^T\W_\sigma^\epsilon dxd\tau\\
& \qquad\quad  \leq
\frac{\mu+\lambda}{2}\int_0^t\|\dv\W_\sigma^\epsilon\|^2d\tau  +
C(\mathcal{N})\int_0^t\|\W_\sigma^\epsilon\|^2d\tau+tC(\mathcal{M}).
\end{align*}

Finally, the above estimates \eqref{nbeef}--\eqref{nbeeh} and the
positivity of $E$ imply \eqref{nbeea}.
\end{proof}

Next, we derive an estimate for
$\|\cu(r_0\u^\epsilon)\|_{\sigma-1}$. Define
\begin{align}\label{factor}
 f^\epsilon(S^\epsilon,\epsilon q^\epsilon):
 =1-\frac{r_0(S^\epsilon)}{r^\epsilon(S^\epsilon, \epsilon q^\epsilon)}.
\end{align}
Hereafter we denote $r_0(t) := r_0(S^\epsilon(t))$ and
$f^\epsilon(t):= f^\epsilon(S^\epsilon(t), \epsilon q^\epsilon(t))$
for notational simplicity.

One can factor out $\epsilon q^\epsilon$ in $f^\epsilon(t)$. In fact,
using Taylor's expansion, one obtains that there exists a smooth
function $g^\epsilon(t)$, such that
\begin{align}\label{nbfe}
f^\epsilon(t)=\epsilon g^\epsilon(t):= \epsilon
g^\epsilon(S^\epsilon(t), \epsilon q^\epsilon(t)), \quad
\|g^\epsilon(t)\|_s\leq  C(\mathcal{M}_\epsilon(T)).
\end{align}
Since
$$\partial_t S^\epsilon + (\u^\epsilon\cdot\nabla)S^\epsilon = \epsilon^2
\frac{\Psi^\epsilon:\nabla\u^\epsilon}{b^\epsilon(S^\epsilon,\epsilon q^\epsilon)},
$$
 the equations for $\u^\epsilon$ are equivalent to
\begin{align}\label{nbff}
[\partial_t + (\u^\epsilon\cdot \nabla)] (r_0\u^\epsilon)
+\frac{1}{\epsilon}\nabla q^\epsilon  = &g^\epsilon\nabla q^\epsilon
+ (1-\epsilon g^\epsilon)( \cu{\H^\epsilon}) \times {\H^\epsilon}
\nonumber\\
&+(1-\epsilon g^\epsilon)\dv \Psi^\epsilon+\epsilon^2  \, r'_0(S^\epsilon)\u^\epsilon \frac{\Psi^\epsilon
 :\nabla\u^\epsilon}{b^\epsilon(S^\epsilon,\epsilon q^\epsilon)}.
\end{align}

We perform the operator \emph{curl} to the equation \eqref{nbff} to obtain that
\begin{align}\label{nbfg}
&  [\partial_t +(\u^\epsilon\cdot \nabla)](\cu(r_0\u^\epsilon))\nonumber\\
 = &
[\u^\epsilon\cdot \nabla, \cu](r_0\u^\epsilon) + \cu [(1-\epsilon
g^\epsilon)( \cu{\H^\epsilon}) \times {\H^\epsilon}]\nonumber\\
& +\cu [(1-\epsilon g^\epsilon)\dv \Psi^\epsilon] + [\cu,
g^\epsilon]\nabla q^\epsilon\nonumber\\
&  +\epsilon^2  \cu \left\{ r'_0(S^\epsilon)\u^\epsilon \frac{\Psi^\epsilon
 :\nabla\u^\epsilon}{b^\epsilon(S^\epsilon,\epsilon q^\epsilon)}\right\}.
  \end{align}

\begin{lem}\label{NLg}
  There exist constants $C_0>0$, $0< \xi_3< \mu$, a function $C(\cdot)$
from $[0,\infty)$ to $[0,\infty)$ and a sufficiently small constant
$\eta_2>0$, such that for all $\epsilon \in (0, 1]$ and all $t\in [0,T]$,
\begin{align}\label{nbga}
\|\{\cu (r_0\u^\epsilon),\cu \H^\epsilon\}(t)\|^2_{s-1}
& + \xi_3\int_0^t\|\nabla\cu\u^\epsilon\|_{s-1}^2d\tau \nonumber\\
\leq & C_0 + tC(\mathcal{N}_\epsilon(T))+\eta_2\int_0^t\|\u^\epsilon\|^2_{s+1}d\tau.
\end{align}
\end{lem}

\begin{proof}
Set  $\mathcal{N}:=\mathcal{N}_\epsilon(T)$
and  $\omega= \cu (r_0\u^\epsilon)$. Taking $\partial^\alpha_x$
$(|\alpha|\leq s-1)$ on \eqref{nbfg}, multiplying
the resulting equations by $\partial^\alpha_x \omega $, and
integrating over $[0,t]\times \R^d$ with $t\leq T$, we obtain
\begin{align}\label{nbgb}
\frac{1}{2}\|\partial^\alpha \omega(t)\|^2
=\,&\frac{1}{2}\|\partial^\alpha \omega(0)\|^2 -\int^t_0\langle(\u^\epsilon\cdot \nabla)
\partial^\alpha  \omega, \partial^\alpha  \omega\rangle(\tau)d\tau\nonumber\\
&+\int^t_0\langle [\u^\epsilon\cdot \nabla, \partial_x^\alpha]\omega, \partial^\alpha
 \omega\rangle(\tau) d\tau\nonumber\\
& +\int^t_0\langle \partial^\alpha \{ [\cu, g^\epsilon]\nabla q^\epsilon\},
\partial^\alpha  \omega\rangle(\tau) d\tau\nonumber\\
& +\int^t_0 \langle \partial^\alpha \{[\u^\epsilon\cdot \nabla, \cu](r_0\u^\epsilon)\},
  \partial^\alpha  \omega\rangle(\tau) d\tau\nonumber\\
& +\int^t_0\left\langle \partial^\alpha \left\{\epsilon^2\, \cu
  \left\{ r'_0(S^\epsilon)\u^\epsilon \frac{\Psi^\epsilon
 :\nabla\u^\epsilon}{b^\epsilon(S^\epsilon, \epsilon q^\epsilon)} \right\}
  \right\}, \partial^\alpha  \omega\right\rangle(\tau) d\tau\nonumber\\
&+\int^t_0 \langle \partial^\alpha \{\cu [(1-\epsilon g^\epsilon)( \cu{\H^\epsilon})
\times {\H^\epsilon}]\}, \partial^\alpha  \omega\rangle(\tau) d\tau \nonumber\\
&+\int^t_0 \langle \partial^\alpha \{\cu [(1-\epsilon g^\epsilon)\dv \Psi^\epsilon]\},
\partial^\alpha  \omega\rangle(\tau) d\tau \nonumber\\
=&:\frac{1}{2}\|\partial^\alpha \omega(0)\|^2+ \int^t_0\sum_{i=1}^7 I_i(\tau).
 \end{align}
We have to estimate the terms $I_i(\tau)$ ($1\leq i\leq 7$) on the right-hand side of (\ref{nbgb}).
Applying partial integrations, we have
\begin{align}
  I_1(\tau)=\int_{\mathbb{R}^d}  |   \partial^\alpha  \omega  |^2 \dv\u^\epsilon
  \leq \|\dv \u^\epsilon(\tau)\|_{L^\infty}\|\partial^\alpha   \omega(\tau)\|^2
  \leq C(\mathcal{N})\|\partial^\alpha
  \omega(\tau)\|^2 ,  \label{nbgc}
\end{align}
while for the term $I_2(\tau)$, an application of Cauchy-Schwarz's inequality gives
\begin{align*}
 | I_2(\tau)|\leq   \|\partial^\alpha  \omega\| \, \|\mathbf{h}_\alpha(\tau)\|,
 \quad  \mathbf{h}_\alpha(\tau):= [\u^\epsilon\cdot \nabla, \partial_x^\alpha]\omega.
\end{align*}
The commutator $\mathbf{h}_\alpha$ is a sum of terms
$\partial^\beta_x \u^\epsilon\partial^\gamma_x \omega$
with multi-indices $\beta$ and $\gamma$ satisfying $|\beta|+|\gamma|\leq s$, $|\beta|>0$,
and $|\gamma|>0$. Thus, the inequality \eqref{nay} with $\sigma=s-1>d/2$ implies that
$\|\mathbf{h}_\alpha(\tau)\|\leq C(\mathcal{N})$. Hence, we have
 \begin{align}\label{nbgd}
 | I_2(\tau)|\leq  C(\mathcal{N})+\|\partial^\alpha  \omega(\tau)\|^2.
\end{align}
Noting that $([\cu, g^\epsilon]\,\mathbf{a}\,)_{i,j} =
a_i\partial_{x_j} g^\epsilon - a_j\partial_{x_i}g^\epsilon$ for $\mathbf{a}=(a_1,\cdots, a_d)$,
the inequality \eqref{nay}, and the estimate \eqref{nbfe},
we can control the term $I_3(\tau)$ as follows
 \begin{align}\label{nbge}
    |I_3(\tau)|& \leq \|\partial^\alpha \{ [\cu, g^\epsilon]\nabla q^\epsilon\}\|\; \|
    \partial^\alpha  \omega\|\nonumber\\
     & \leq K \| [\cu, g^\epsilon]\nabla q^\epsilon\|_{s-1}\, \| \partial^\alpha  \omega\|\nonumber\\
    & \leq K \|\nabla g^\epsilon(\tau)\|_{s-1} \|\nabla q^\epsilon(\tau)\|_{s-1}\, \|
    \partial^\alpha  \omega\|\nonumber\\
    &\leq  C(\mathcal{N})+\|\partial^\alpha  \omega(\tau)\|^2.
 \end{align}

Similarly, the term $I_4(\tau)$ can be bounded as follows.
 \begin{align}\label{nbgf}
    |I_4(\tau)|& \leq K\|\partial^\alpha \{[\u^\epsilon\cdot \nabla, \cu](r_0\u^\epsilon)\}\|
    \; \| \partial^\alpha  \omega\|\nonumber\\
& \leq K\|[\u^\epsilon\cdot \nabla, \cu](r_0\u^\epsilon)\|_{s-1}\, \| \partial^\alpha  \omega\|\nonumber\\
      & \leq K\|[\u^\epsilon_j, \cu]\partial_{x_j}(r_0\u^\epsilon)\|_{s-1}\,
      \| \partial^\alpha  \omega\|\nonumber\\
        &\leq  C(\mathcal{N})+\|\partial^\alpha  \omega(\tau)\|^2.
 \end{align}

To bound the term $I_5(\tau)$, we use the Moser-type inequality (see \cite{KM}) to deduce
\begin{align}\label{nbgg}
|I_5(\tau)| \leq & \epsilon^2 K \left\|\partial^\alpha \left\{ \cu
\left\{ r'_0(S^\epsilon)\u^\epsilon \frac{\Psi^\epsilon
 :\nabla\u^\epsilon}{b^\epsilon(S^\epsilon,
  \epsilon q^\epsilon)}\right\}
  \right\}\right\|\cdot \| \partial^\alpha  \omega\|\nonumber\\
=& \epsilon^2 K\left\|\partial^\alpha \left[\left(
\frac{r'_0(S^\epsilon)\Psi^\epsilon
 :\nabla\u^\epsilon}{b^\epsilon(S^\epsilon,
  \epsilon q^\epsilon)}\right) \cu\u^\epsilon\right]\right\|\cdot \| \partial^\alpha  \omega\|\nonumber\\
  &   + \epsilon^2 K\left\|\partial^\alpha \left[\nabla \left( \frac{r'_0(S^\epsilon)\Psi^\epsilon
 :\nabla\u^\epsilon}{b^\epsilon(S^\epsilon,
  \epsilon q^\epsilon)}\right) \times \u^\epsilon\right]\right\|\cdot \| \partial^\alpha  \omega\|\nonumber\\
 \leq & \epsilon^2 K  \|\cu\u^\epsilon\|_{L^\infty}\left\|D^{s-1}\Big(\frac{r'_0(S^\epsilon)\Psi^\epsilon
 :\nabla\u^\epsilon}{b^\epsilon(S^\epsilon,
  \epsilon q^\epsilon)}\Big)\right\|\cdot \| \partial^\alpha  \omega\|\nonumber\\
 &  +\epsilon^2 K\|D^{s-1}(\cu\u^\epsilon)\|\left\|\frac{r'_0(S^\epsilon)\Psi^\epsilon
 :\nabla\u^\epsilon}{b^\epsilon(S^\epsilon,
  \epsilon^2 q^\epsilon)}\right\|_{L^\infty}\cdot \| \partial^\alpha  \omega\|\nonumber\\
    &+ \epsilon^2 K  \|\u^\epsilon\|_{L^\infty}\left\|D^{s}\Big(\frac{r'_0(S^\epsilon)\Psi^\epsilon
 :\nabla\u^\epsilon}{b^\epsilon(S^\epsilon,
  \epsilon^2 q^\epsilon)}\Big)\right\|\cdot \| \partial^\alpha  \omega\|\nonumber\\
 &  +\epsilon^2 K\|D^{s-1}\u^\epsilon\|\left\|\nabla\left(\frac{r'_0(S^\epsilon)\Psi^\epsilon
 :\nabla\u^\epsilon}{b^\epsilon(S^\epsilon,
  \epsilon^2 q^\epsilon)}\right)\right\|_{L^\infty}\cdot \| \partial^\alpha  \omega\|\nonumber\\
\leq & C(\mathcal{N})+\epsilon^2C(\mathcal{N})\|\u^\epsilon\|_{s+1}^2+\|\partial^\alpha
\omega(\tau)\|^2 ,
\end{align}
where the condition $s> 2+d/2$ and the inequality \eqref{naz} have been used.
For the term $I_6(\tau)$, by virtue of \eqref{va},
$\cu \cu \mathbf{a} =\nabla \,\dv \,\mathbf{a} -\Delta \mathbf{a}$. Thus, we
 integrate by parts to see that
\begin{align*}
   I_6(\tau)   = &  \langle \partial^\alpha \{(1-\epsilon g^\epsilon)
(\cu{\H^\epsilon})\times {\H^\epsilon}\},\partial^\alpha\{\cu\cu( r_0\u^\epsilon)\}\rangle
\nonumber\\
= & \langle\partial^\alpha\{(1-\epsilon g^\epsilon)( \cu{\H^\epsilon})\times {\H^\epsilon}\},
\partial^\alpha \{\nabla \dv ( r_0\u^\epsilon) -\Delta  (r_0\u^\epsilon)\} \rangle ,\nonumber
\end{align*}
and use Cauchy-Schwarz's inequality and \eqref{naz} to conclude
\begin{align}\label{nbggg}
|I_6(\tau)|\leq C(\mathcal{N})+ \theta_{1}\|\u^\epsilon(\tau)\|^2_{s+1},
\end{align}
where $\theta_{1}>0$ is a sufficiently small constant independent of $\epsilon$.
Next, we deal with the term $I_7(\tau)$. By the vector
identities  and integration by parts, we see that there exists a
sufficiently small $\theta_2$, such that
\begin{align}\label{nbgggi}
I_7(\tau)\leq & -\inf \{r_0(S^\epsilon)\}\|\nabla\cu
\u^\epsilon(\tau)\|_{\sigma-1}+C(\mathcal{N})\nonumber\\
& +\theta_2\|\u^\epsilon(\tau)\|_{s+1}+\epsilon
C(\mathcal{N})\|\u^\epsilon(\tau)\|_{s+1}.
\end{align}

Finally, to estimate $\|\cu \H^\epsilon\|_{s-1}$, we apply the operator \emph{curl} to
\eqref{nam} and use the vector identity \eqref{naff} to obtain
\begin{align}\label{nbgh}
 & \partial_t (\cu \H^\epsilon)+\u^\epsilon\cdot\nabla(\cu\H^\epsilon)\nonumber\\
 & = -[\cu, \u^\epsilon]\cdot\nabla
\H^\epsilon+\cu((\H^\epsilon\cdot\nabla)\u^\epsilon-\H^\epsilon\dv\u^\epsilon ).
\end{align}
By the commutator inequality and Sobolev's inequalities, we find that
\begin{align*}
\|[\cu, \u^\epsilon]\cdot\nabla\H^\epsilon\| \leq C(\mathcal{N} )
\end{align*}
and
\begin{align*}
\|\cu ( (\H^\epsilon\cdot\nabla)\u^\epsilon - \H^\epsilon\dv\u^\epsilon )\|
\leq C(\mathcal{N})+\theta_3\|\u^\epsilon\|_{s+1}^2
\end{align*}
for sufficiently small constant $\theta_3>0$. Then, by
arguments similar to those used in Lemma \ref{NLa}, we derive that
\begin{align}\label{nbgggk}
\|\cu \H^\epsilon\|_{s-1}^2\leq  C_0 +
tC(\mathcal{N})+\theta_3\int_0^t \|\u^\epsilon(\tau)\|_{s+1}^2d\tau.
\end{align}

  Thus, the lemma follows from
  adding up the estimates \eqref{nbgb}--\eqref{nbgggk} for all $|\alpha|\leq s-1$
   and choosing constants $\theta_1$, $\theta_2$ and $\theta_3$ appropriately small.
   \end{proof}

Next we complete the proof of Theorem \ref{prop} by the following
estimate.

\begin{lem}\label{NLk}
There exist   constants $C_0>0$, $0<\xi_4<\mu$, and  a function
$C(\cdot)$ from $[0,\infty)$ to $[0,\infty)$, such that for all
$\epsilon \in (0,1]$ and $t\in [0,T]$,
\begin{align}\label{nbja}
\|(S^\epsilon,q^\epsilon,\u^\epsilon,\H^\epsilon)(t)\|^2_{s}+ \xi_4
\int_0^t\|\u^\epsilon\|^2d\tau\leq
C_0+(t+\epsilon)C(\mathcal{M}_\epsilon(T)).
\end{align}
\end{lem}

\begin{proof}
First, from \eqref{nbha} we get
\begin{align}\label{nbha2}
\|\u^\epsilon\|_{\sigma+1}^2 \leq \tilde{
C}\Big\{\|\nabla(\{L_{E^\epsilon}(\partial_x)\}^\sigma
{\W}^\epsilon\|_{0}^2
+\|\nabla\cu(r_0\u^\epsilon)\|_{\sigma-1}^2+\|\u^\epsilon\|_{\sigma}^2\Big\},
\end{align}
where $\tilde{C}:=C_1+tC(\mathcal{M}_\epsilon(T))+\epsilon
C(\mathcal{M}_\epsilon(T))$. Moreover, using Lemma \ref{NLa}, we obtain
\begin{align}\label{nbha3}
\|\u^\epsilon\|_{\sigma+1}^2 \leq \tilde{
C}\Big\{\|\nabla\{L_{E^\epsilon}(\partial_x)\}^\sigma
{\W}^\epsilon\|_{0}^2
+\|\nabla(\cu\u^\epsilon)\|_{\sigma-1}^2+\|\u^\epsilon\|_{\sigma}^2\Big\}.
\end{align}
In view of \eqref{nbha3}, there exists a constant $\kappa_2$ such that
\begin{align}
\frac{\xi_2}{2}\|\nabla\W_\sigma^\epsilon\|_0^2+\xi_3\|\nabla\cu\u^\epsilon\|_{\sigma-1}^2 \geq &
\kappa_2\|\u^\epsilon\|_{\sigma+1}^2-\tilde{C}_1\Big\{\|\nabla\{L_{E^\epsilon}(\partial_x)\}^\sigma
{\W}^\epsilon\|_{0}^2\nonumber\\
&+\|\nabla(\cu\u^\epsilon)\|_{\sigma-1}^2\Big\}-\tilde{C}\|\u^\epsilon\|_{\sigma}^2, \label{SM1}
\end{align}
where $\tilde{C}_1=tC(\mathcal{M}_\epsilon(T))+\epsilon
C(\mathcal{M}_\epsilon(T))$. Now, we combine the estimates
\eqref{nbeea} and \eqref{nbga} with \eqref{SM1}, and use the fact that
$\dv\H=0$ to conclude that there exists a positive constant $\kappa_3$,
such that
\begin{align*}
&\|(L_E(\partial_x))^\sigma\U^\epsilon\|_0^2+\|\cu(r_0\u^\epsilon)\|_{\sigma-1}^2+\kappa_3\int_0^t
\|\u^\epsilon\|_{\sigma+1}^2dx\\
\leq& C_0+tC(\mathcal{M}_\epsilon(T))+
\tilde{C}\int_0^t\|\u^\epsilon\|_{\sigma}^2d\tau\\
& +\tilde{C}_1\int_0^t\Big\{\|\nabla\{L_{E^\epsilon}(\partial_x)\}^\sigma
{\W}^\epsilon\|_{0}^2+\|\nabla(\cu\u^\epsilon)\|_{\sigma-1}^2\Big\}d\tau\\
\leq& C_0+(t+\epsilon)C(\mathcal{M}_\epsilon(T))+\tilde{C}\int_0^t\|\u^\epsilon\|_{\sigma}^2d\tau
\end{align*}
for sufficiently small $\eta_1$ and $\eta_2$. Thus by induction, we conclude that
\begin{align*}
 \|\{L_{E^\epsilon}(\partial_x)\}^\sigma\U^\epsilon\|_0^2
 & +\|\cu(r_0\u^\epsilon)\|_{\sigma-1}^2\nonumber\\
&  +\kappa_3\int_0^t \|\u^\epsilon\|_{\sigma+1}^2d\tau
 \leq C_0+(t+\epsilon)C(\mathcal{M}_\epsilon(T)).
\end{align*}
Using \eqref{nbha} again, we obtain the estimate \eqref{nbja} by induction on
$\sigma \in\{0,\dots, s\}$.
\end{proof}

\section{Incompressible limit}

In this section, we shall prove the convergence part of Theorem
\ref{mth} by modifying the method developed by M\'{e}tivier and
Schochet \cite{MS01}, see also some extensions \cite{A05,A06, LST}.

\begin{proof}[Proof of the convergence part of Theorem \ref{mth}]
 The uniform bound \eqref{nav} implies that, after extracting a subsequence,
 one gets the following limits:
   \begin{align}
      & (q^\epsilon,\u^\epsilon,\H^\epsilon )\rightharpoonup (q,\v,\bar\H )\quad \text{weakly-}\ast
     \  \text{in} \quad L^\infty (0,T; H^s(\mathbb{R}^d)).
\label{caa}
\end{align}

 The equations \eqref{nam} and \eqref{nan} imply that
 $\partial_tS^\epsilon$ and $\partial_t \H^\epsilon \in C([0,T],H^{s-1}(\mathbb{R}^d))$.
 Thus, after further extracting a subsequence, we
  obtain that, for all $s'<s$,
\begin{align}
 &     S^\epsilon \rightarrow \bar S \quad \text{strongly in}  \quad
 C([0,T],H^{s'}_{\mathrm{loc}}(\mathbb{R}^d)),\label{cab}\\
&\H^\epsilon  \rightarrow \bar \H \quad \text{strongly in} \quad
C([0,T],H^{s'}_{\mathrm{loc}}(\mathbb{R}^d)),\label{cad}
\end{align}
where the limit $\bar \H \in C([0,T],H^{s'}_{\mathrm{loc}}(\mathbb{R}^d))\cap
L^\infty (0,T;H^{s}_{\mathrm{loc}}(\mathbb{R}^d))$. Similarly, by
\eqref{nbfg} and the uniform bound \eqref{nav}, we have
\begin{align}
 & \cu (r_0(S^\epsilon) \u^\epsilon) \rightarrow \cu (r_0(\bar S) \v) \quad
 \text{strongly in}  \quad C([0,T],H^{s'-1}_{\mathrm{loc}}(\mathbb{R}^d))\label{cac}
  \end{align}
for all $ s'<s$, where $ r_0(\bar S)=\lim_{\epsilon \rightarrow
0}r_0(S^\epsilon):= \lim_{\epsilon \rightarrow
0}r^\epsilon(S^\epsilon,0)$.

In order to obtain the limit system, we need to prove that the
convergence in \eqref{caa} holds in the strong topology of
$L^2(0,T;H^{s'}_{\mathrm{loc}}(\mathbb{R}^d))$ for all $s'<s$. To this end,
we first show that $q=0$ and $\dv \v=0$. In fact, from \eqref{nbeb} we get
\begin{align}\label{nbeb1}
 \epsilon E^\epsilon(S^\epsilon,\epsilon q^\epsilon)\partial_t\U^\epsilon
 +L(\partial_x)\U^\epsilon=-\epsilon E^\epsilon(S^\epsilon,\epsilon q^\epsilon)
 \u^\epsilon\cdot \nabla\U^\epsilon+\epsilon (\mathbf{J}^\epsilon+\mathbf{V}^\epsilon).
\end{align}
Since
\begin{align*}
 E^\epsilon(S^\epsilon,\epsilon
 q^\epsilon)-E^\epsilon(S^\epsilon,0)=O(\epsilon),
\end{align*}
we have
\begin{align}\label{pass}
\epsilon E^\epsilon(S^\epsilon,0)\partial_t\U^\epsilon
 +L(\partial_x)\U^\epsilon=\epsilon \mathbf{h}^\epsilon ,
\end{align}
where $\mathbf{h}^\epsilon$ is uniformly bounded in $C([0,T],H^{s-2}(\mathbb{R}^d))$
in view of \eqref{nav}. Passing to the weak
limit to  \eqref{pass}, we obtain $\nabla q=0$ and $\dv \v=0$.
Since $q\in L^\infty(0,T;H^s(\mathbb{R}^d))$, we infer that $q=0$.
Noticing that the strong compactness for the incompressible
components by \eqref{cac}, it is sufficient to prove the following
proposition on the acoustic components.

\begin{prop}\label{LC}
Suppose that the assumptions in Theorem \ref{mth} hold, then
$q^\epsilon$ converges strongly to $0$ in
$L^2(0,T;H^{s'}_{\mathrm{loc}}(\mathbb{R}^d))$ for all $s'<s$, and $\dv
\u^\epsilon$ converges strongly to $0$ in $L^2(0,T;
H^{s'-1}_{\mathrm{loc}}(\mathbb{R}^d))$ for all $s'<s$.
\end{prop}

The proof of Proposition \ref{LC} is built on the the following
dispersive estimates on the wave equations obtained by M\'{e}tivier
and Schochet \cite{MS01} and reformulated in \cite{A06}.

\begin{lem}[\cite{MS01,A06}]\label{LD}
   Let $T>0$ and $w^\epsilon$ be a bounded sequence in $C([0,T],H^2(\mathbb{R}^d))$, such that
   \begin{align*}
    \epsilon^2\partial_t(a^\epsilon \partial_t w^\epsilon)-\nabla\cdot (b^\epsilon \nabla w^\epsilon)=c^\epsilon,
   \end{align*}
where $c^\epsilon$ converges to $0$ strongly in $L^2(0,T;
L^2(\mathbb{R}^d))$. Assume further that, for some $s> d/2+1$, the
coefficients $(a^\epsilon,b^\epsilon)$ are uniformly bounded in
$C([0,T];H^s(\mathbb{R}^d))$ and converges in
$C([0,T];H^s_{\mathrm{loc}}(\mathbb{R}^d))$ to a limit $(a,b)$
satisfying the decay estimate
\begin{gather*}
   |a(x,t)-\underline a|\leq C_0 |x|^{-1-\delta}, \quad |\nabla_x a(x,t)|\leq C_0 |x|^{-2-\delta}, \\
 |b(x,t)-\underline b|\leq C_0 |x|^{-1-\delta}, \quad |\nabla_x b(x,t)|\leq C_0 |x|^{-2-\delta},
\end{gather*}
for some given positive constants $\underline a$, $\underline b$, $
C_0$ and $\delta$. Then the sequence $w^\epsilon$ converges to $0$
in  $L^2(0,T; L^2_{\mathrm{loc}}(\mathbb{R}^d))$.
\end{lem}

\begin{proof}[Proof of Proposition \ref{LC}]
We first show that $q^\epsilon$ converges strongly to $0$ in
$L^2(0,T; \linebreak H^{s'}_{\mathrm{loc}}(\mathbb{R}^d))$ for all
$s'<s$. An application of the operator $\epsilon^2 \partial_t$ to \eqref{nak}
gives
\begin{align}\label{cae}
\epsilon^2\partial_t (a^\epsilon(S^\epsilon,\epsilon q^\epsilon)
\partial_t q^\epsilon) +\epsilon \partial_t\dv \u^\epsilon
=-\epsilon^2 \partial_t\{a^\epsilon(S^\epsilon,\epsilon
q^\epsilon)(\u^\epsilon\cdot\nabla )q^\epsilon\}.
\end{align}
Dividing \eqref{nal} by $r^\epsilon(S^\epsilon,\epsilon q^\epsilon)$
and then taking the operator \emph{div} to the resulting equation, one has
\begin{align}\label{caf}
& \partial_t \dv \u^\epsilon +\frac{1}{\epsilon}\dv
\Big(\frac{1}{r^\epsilon(S^\epsilon,\epsilon q^\epsilon)}\nabla q^\epsilon\Big)\nonumber\\
& \qquad \qquad =-\dv ((\u^\epsilon\cdot \nabla)\u^\epsilon)+\dv
\Big(\frac{1}{r^\epsilon(S^\epsilon,\epsilon q^\epsilon)}(\cu
\H^\epsilon)\times \H^\epsilon\Big)\nonumber\\
&\qquad\qquad\quad+\dv\Big(\frac{1}{r^\epsilon(S^\epsilon,\epsilon
q^\epsilon)}\dv\Psi(\u^\epsilon)\Big).
\end{align}
Subtracting \eqref{caf} from \eqref{cae}, we obtain
\begin{align}\label{cag}
\epsilon^2\partial_t (a^\epsilon(S^\epsilon,\epsilon q^\epsilon)
\partial_t q^\epsilon) -\dv
\Big(\frac{1}{r^\epsilon(S^\epsilon,\epsilon q^\epsilon)}\nabla
q^\epsilon\Big) =  F^\epsilon(S^\epsilon, q^\epsilon, \u^\epsilon,
\H^\epsilon),
\end{align}
where
\begin{align*}
 F^\epsilon(S^\epsilon, q^\epsilon, \u^\epsilon, \H^\epsilon)
  =\, & \epsilon  \dv \Big(\frac{1}{r^\epsilon(S^\epsilon,\epsilon q^\epsilon)}
  (\cu \H^\epsilon)\times \H^\epsilon\Big)\\
& +\epsilon\dv\Big(\frac{1}{r^\epsilon(S^\epsilon,\epsilon
q^\epsilon)}\dv\Psi(\u^\epsilon)\Big) -\epsilon \dv
((\u^\epsilon\cdot \nabla)\u^\epsilon)\nonumber\\
 & - \epsilon^2
\partial_t\{a^\epsilon(S^\epsilon,\epsilon
q^\epsilon)(\u^\epsilon\cdot\nabla )q^\epsilon\}.
\end{align*}

In view of the uniform boundedness of $(S^\epsilon,q^\epsilon,\u^\epsilon,\H^\epsilon)$,
the smoothness and positivity assumptions on $a^\epsilon$ and $r^\epsilon$,
 and the convergence of $S^\epsilon$, we find that
\begin{align*}
F^\epsilon(S^\epsilon, q^\epsilon, \u^\epsilon, \H^\epsilon)
\rightarrow 0 \quad \text{strongly in} \quad L^2(0,T;
L^2(\mathbb{R}^d)),
\end{align*}
and the coefficients in \eqref{cag} satisfy the requirements in
Lemma \ref{LD}. Therefore, by virtue of Lemma \ref{LD},
\begin{align*}
q^\epsilon \rightarrow 0 \quad \text{strongly in} \quad L^2(0,T;
L^2_{\mathrm{loc}}(\mathbb{R}^d)).
\end{align*}

On the other hand, the uniform boundedness of $q^\epsilon$ in $C([0,T],H^s(\mathbb{R}^d))$
and an interpolation argument yield that
\begin{align*}
q^\epsilon \rightarrow 0 \quad \text{strongly in}
 \quad    L^2(0,T;  H^{s'}_{\mathrm{loc}}(\mathbb{R}^d))\ \ \text{for all} \ \  s'<s.
\end{align*}
Similarly, we can obtain the  convergence of  $\dv u^\epsilon$.
\end{proof}

We continue our proof of Theorem \ref{mth}. From Proposition
\ref{LC}, we know that
\begin{align*}
\dv\, \u^\epsilon \rightarrow \dv\,\v\quad \mathrm{in} \quad L^2(0,T;
H^{s'-1}_{\mathrm{loc}}(\mathbb{R}^d)).
\end{align*}
Hence, from \eqref{cac} it follows that
\begin{align*}
 \u^\epsilon \rightarrow \v\quad \mathrm{in} \quad L^2(0,T;
H^{s'}_{\mathrm{loc}}(\mathbb{R}^d))\qquad\mbox{for all }s'<s.
\end{align*}
By \eqref{cab}, \eqref{cad} and Proposition \ref{LD}, we obtain
\begin{equation*}
\begin{array}{ccl}
r^\epsilon(S^\epsilon, \epsilon q^\epsilon) \rightarrow r_0(\bar S)
& \mathrm{in} &  L^\infty(0,T; L^\infty(\mathbb{R}^d));\\
\nabla \u^\epsilon\rightarrow\nabla\v & \mathrm{in}
& L^2(0,T; H^{s'-1}_{\mathrm{loc}}(\mathbb{R}^d));\\
\nabla \H^\epsilon\rightarrow \nabla  \bar \H    & \mathrm{in} &
L^2(0,T; H^{s'-1}_{\mathrm{loc}}(\mathbb{R}^d)).
\end{array}
\end{equation*}
Passing to the limit in the equations for $S^\epsilon$ and
$\H^\epsilon$, we see that the limits $\bar{S}$ and $\bar \H$
satisfy
\begin{align*}
  \partial_t \bar{S} +(\v \cdot \nabla) \bar{S} =0, \quad
   \partial_t \bar{\H}  + ( {\v}  \cdot \nabla) \bar{\H}
   - ( \bar{\H} \cdot \nabla) {\v}  =0
\end{align*}
in the sense of distributions. Since   $r^\epsilon(S^\epsilon,
\epsilon q^\epsilon)-r_0(S^\epsilon)=O(\epsilon)$, we have
\begin{align*}
(r^\epsilon(S^\epsilon, \epsilon
q^\epsilon)-r_0(S^\epsilon))(\partial_t
\u^\epsilon+(\u^\epsilon\cdot \nabla)\u^\epsilon)\rightarrow 0 ,
\end{align*}
whence,
\begin{align*}
r^\epsilon ( S^\epsilon, \epsilon q^\epsilon)(\partial_t
\u^\epsilon +(\u^\epsilon\cdot \nabla)\u^\epsilon )
 =\,& (r^\epsilon(S^\epsilon, \epsilon q^\epsilon)-r_0(S^\epsilon))(\partial_t
\u^\epsilon+(\u^\epsilon\cdot \nabla)\u^\epsilon)\nonumber\\
& +\partial_t(r_0(S^\epsilon)\u^\epsilon)+(\u^\epsilon \cdot \nabla)(r_0(S^\epsilon)\u^\epsilon)\\
\rightarrow &\, r_0(\bar S)(\partial_t \v +(\v\cdot \nabla)\v)
\end{align*}
in the sense of distributions.

Applying the operator \emph{curl} to
the momentum equation \eqref{nal} and taking to the limit, we find that
\begin{align*}
\cu\big( r_0(\bar{S})(\partial_t \v+\v\cdot \nabla \v)
 -(\cu\bar{\H}) \times \bar{\H}-\mu\Delta \v \big)=0.
\end{align*}
Therefore, by the fact that $\cu \nabla =0$, the limit
$(\bar S,\v ,\bar \H)$ satisfies
\begin{align}
&   r(\bar{S},0)(\partial_t \v+(\v\cdot \nabla) \v)
  -(\cu\bar{\H}) \times \bar{\H}-\mu\Delta \v +\nabla \pi =0,  \label{cba} \\
&  \partial_t \bar{\H}  + ( {\v}  \cdot \nabla) \bar{\H}
   - ( \bar{\H} \cdot \nabla) {\v} =0,  \label{cbb} \\
 &\partial_t \bar{S} +(\v \cdot \nabla) \bar{S} =0,  \label{cbc} \\
& \dv \v=0,  \quad \dv \bar{\H} =0\label{cbd}
\end{align}
for some function $\pi$.

If we employ the same arguments as in the proof
of Theorem 1.5 in \cite{MS01}, we find that $(\bar S, \v, \bar
\H)$ satisfies the initial conditions \eqref{nax}. Moreover, the
standard iterative method shows that the system \eqref{cba}--\eqref{cbd}
with initial data \eqref{nax} has a unique solution
$(S^*, \v^*, \H^*)\in C([0,T],H^s(\mathbb{R}^d)).$ Thus, the uniqueness of
solutions to the limit system \eqref{cba}--\eqref{cbd} implies that
the above convergence results hold for
the full sequence of $(S^\epsilon, q^\epsilon, \u^\epsilon,
\H^\epsilon)$. Thus, the proof is  completed.
\end{proof}

\section{Compressible non-isentropic MHD equations with infinite Reynolds number}

In the study of magnetohydrodynamics, for some
local processes in the cosmic system, the effect of the magnetic diffusion will become
very important, see \cite{Hu87}. Moreover, when the Reynolds number
of a fluid is very high and the temperature changes very slowly,
we can ignore the viscosity and the heat conductivity of the fluid in the MHD equations.
 In such situation, the compressible MHD equations in the non-isentropic case take the form:
\begin{align}
&\partial_t\rho +\dv(\rho\u)=0, \label{rnaa} \\
&\partial_t(\rho\u)+\dv\left(\rho\u\otimes\u\right)+\nabla p
  =(\cu \H)\times \H, \label{rnab} \\
&\partial_t\E+\dv\left(\u(\E'+p)\right) =\dv\big[(\u\times\H)\times\H+\nu \H\times (\cu \H)\big],
  \label{rnac} \\
&\partial_t\H-\cu(\u\times\H)=-\cu (\nu\,\cu\H),\quad
\dv\H=0.\label{rnad}
\end{align}
As before here  $\rho$ denotes the density, $\u\in \R^d$ ($d=2,3$) the velocity,
$\H\in \R^d$ the magnetic field; $\E$ the total energy given by
$\E=\E'+|\H|^2/2$  and $\E'=\rho\left(e+|\u|^2/2\right)$ with $e$ being
the internal energy, $\rho|\u|^2/2$ the kinetic energy, and
$|\H|^2/2$ the magnetic energy. The equations of state
$p=p(\rho,\theta)$ and  $e=e(\rho,\theta)$ relate the pressure $p$
and the internal energy $e$ to the density $\rho$ and the temperature
$\theta$. The constant  $\nu>0$ is the magnetic
diffusivity acting as a magnetic diffusion coefficient of the
magnetic field.

Using the Gibbs relation \eqref{gibbs}
and the  identities \eqref{naeo} and \eqref{nae},
the equation of energy conservation \eqref{rnac} can be replaced by
\begin{equation}\label{rnaf}
\partial_t(\rho S)+\dv(\rho  S\u)=\frac{\nu}{\theta}|\cu \H|^2,
\end{equation}
where $S$ denotes the entropy.

As in Section 1, we reconsider the equations of state
as functions of $S$ and $p$, i.e., $\rho =R(S,p)$ and
$\theta=\Theta(S,p)$ for some positive smooth functions $R$ and
$\Theta$ defined for all $S$ and $p>0$, and satisfying
$\partial R/\partial p >0$.  Then, by utilizing \eqref{naa}
together with the constraint $\dv {\H}=0$, the system
\eqref{rnaa}, \eqref{rnab}, \eqref{rnad} and \eqref{rnaf} can be rewritten as
\begin{align}
    & A(S,p)(\partial_t p+(\u\cdot \nabla) p)+\dv \u=0,\label{rnag}\\
& R(S,p)(\partial_t \u+(\u\cdot \nabla) \u)+\nabla p = (\cu \H)\times \H, \label{rnah}\\
&  \partial_t {\H} -\cu(\u\times\H)=-\cu (\nu\,\cu\H), \quad \dv \H=0, \label{rnai}\\
 & R(S,p)\Theta(S,p)(\partial_tS+(\u\cdot \nabla) S)={\nu}|\cu \H|^2,\label{rnaj}
\end{align}
where $A(S,p)$ is defined by \eqref{asp}.
By introducing the dimensionless parameter $\epsilon$,
 and making the following changes of variables:
\begin{gather*}
    p (x, t)=p^\epsilon (x,\epsilon t), \quad S (x, t)=S^\epsilon (x,\epsilon t), \\
    {\u} (x,t)=\epsilon \u^\epsilon(x,\epsilon t), \;\;\;
   {\H} (x,t)=\epsilon \H^\epsilon(x,\epsilon t),\;\;\;
   \nu=\epsilon\,\mu^\epsilon,
 \end{gather*}
and $p^\epsilon (x, \epsilon t)=\underline{p} e^{\epsilon
q^\epsilon(x,\epsilon t)}$ for some positive constant $\underline{p}$, the system
 \eqref{rnag}--\eqref{rnaj} can be rewritten as
\begin{align}
& a^\epsilon(S^\epsilon,\epsilon q^\epsilon)(\partial_t q^\epsilon+(\u^\epsilon\cdot\nabla)q^\epsilon)
    +\frac{1}{\epsilon}\dv \u^\epsilon=0,\label{rnak}\\
& r^\epsilon(S^\epsilon,\epsilon q^\epsilon)(\partial_t
\u^\epsilon+(\u^\epsilon\cdot \nabla)
\u^\epsilon)+\frac{1}{\epsilon}\nabla q^\epsilon
 =  ( \cu{\H^\epsilon}) \times {\H^\epsilon},  \label{rnal}\\
&  \partial_t {\H}^\epsilon -\cu(\u^\epsilon\times\H^\epsilon)-\mu^\epsilon \Delta \H^\epsilon=0,
\quad \dv \H^\epsilon=0, \label{rnam}\\
 &b^\epsilon(S^\epsilon,\epsilon q^\epsilon)(\partial_tS^\epsilon+(\u^\epsilon\cdot \nabla)S^\epsilon)
 =\epsilon^2{\mu}^\epsilon|\cu \H^\epsilon|^2,\label{rnan}
\end{align}
where we have used the identity
$
 \cu \cu \H^\epsilon=\nabla \dv \H^\epsilon-\Delta \H^\epsilon, $
  the constraint $\dv \H^\epsilon=0$,  and the abbreviations \eqref{nann} and \eqref{nano}.

Formally, we obtain from \eqref{rnak} and \eqref{rnal} that $\nabla
q^\epsilon \rightarrow 0$ and $\dv \u^\epsilon=0$ as $\epsilon
\rightarrow 0$. Applying the operator \emph{curl} to \eqref{rnal},
using the fact that $\cu \nabla =0$,  and letting $\epsilon\rightarrow 0$, we obtain
\begin{align*}
\cu\big( r(\bar{S},0)(\partial_t \v+\v\cdot \nabla \v)
 -(\cu\bar{\H}) \times \bar{\H} \big)=0,
\end{align*}
where we have assumed that
$(S^\epsilon,q^\epsilon,\u^\epsilon,\H^\epsilon)$ and
$r^\epsilon(S^\epsilon,\epsilon q^\epsilon)$ converge to
$(\bar{S},0,\v,\bar{\H})$ and $r(\bar S,0)$ in some sense, respectively.
Finally, Letting $\mu^\epsilon\rightarrow \mu>0$ and applying the identity \eqref{naff},
we expect to get the following incompressible MHD equations:
\begin{align}
&   r(\bar{S},0)(\partial_t \v+(\v\cdot \nabla) \v)
  -(\cu\bar{\H}) \times \bar{\H} +\nabla \hat\pi =0, \label{rnao} \\
&  \partial_t \bar{\H}  + ( {\v}  \cdot \nabla) \bar{\H}
   - ( \bar{\H} \cdot \nabla) {\v} -\mu\Delta \bar\H =0, \label{rnap} \\
 &\partial_t \bar{S} +(\v \cdot \nabla) \bar{S} =0, \label{rnaq}\\
& \dv\,\v=0,  \quad \dv \bar{\H} =0  \label{rnar}
\end{align}
for some function $\hat\pi$.

We supplement the system \eqref{nak}--\eqref{nan} with initial conditions
\begin{align}
(S^\epsilon,q^\epsilon,\u^\epsilon,\H^\epsilon)|_{t=0}
=(S^\epsilon_0,q^\epsilon_0,\u^\epsilon_0,\H^\epsilon_0). \label{rnas}
\end{align}
The main result of this section reads as follows.
\begin{thm}\label{NTa}
 Let $s> d/2+2$ be an integer and $\mu^\epsilon \rightarrow \mu>0$.
 For any constant $M_0>0$, there is a positive constant
 $T=T(M_0)$, such that for all $\epsilon \in (0,1]$ and any initial data
$(S^\epsilon_0,q^\epsilon_0,\u^\epsilon_0,\H^\epsilon_0)$ satisfying
\begin{align}\label{rnat}
\|(S^\epsilon_0,q^\epsilon_0,\u^\epsilon_0,\H^\epsilon_0)\|_{H^s(\R^d)}\leq M_0,
\end{align}
the Cauchy problem \eqref{rnak}--\eqref{rnan}, \eqref{rnas} has a unique
solution $(S^\epsilon,q^\epsilon,\u^\epsilon,\H^\epsilon)\in C^0([0,T],H^s(\R^d))$,
satisfying that for all $\epsilon \in (0,1]$ and $t\in [0,T]$,
\begin{gather}
 \|\big(S^\epsilon,q^\epsilon,\u^\epsilon,\H^\epsilon\big)(t)\|_{H^s(\R^d)}\leq N
 \quad\mbox{for some constant }N=N(M_0)>0. \label{rnau}
\end{gather}
Moreover, suppose further that the initial data
$(S^\epsilon_0,q^\epsilon_0,\u^\epsilon_0,\H^\epsilon_0)$ converge strongly
in $H^s(\R^d)$ to $(S_0, 0,\v_0, \H_0)$ and $S^\epsilon_0 $ decays sufficiently
rapidly at infinity in the sense that
\begin{equation}\label{rnaw}
|S^\epsilon_0(x)-\underline{S}\,\, |\leq   {N}_0 |x|^{-1-\iota}, \quad
|\nabla S^\epsilon_0(x)|\leq  N_0 |x|^{-2-\iota}
\end{equation}
for all $\epsilon  \in (0,1]$ and some positive constants
$\underline S$, $N_0$ and $\iota$.
  Then, $(S^\epsilon,q^\epsilon,\u^\epsilon,\H^\epsilon)$ converges weakly in
  $L^\infty(0,T; H^s(\R^d))$
and strongly in $L^2(0,T;H^{s'}_{\mathrm{loc}}(\R^d))$ to a limit
$(\bar S,0,\v,\bar\H)$ for all $s'<s$. Moreover, $(\bar S,\v,\bar{\H})$ is the unique
  solution in $C([0,T],$ $H^s(\R^d))$ to the system \eqref{rnao}--\eqref{rnar} with initial
data $(\bar S,\v,\bar{\H})|_{t=0}=(S_0,\w_0,\H_0)$, where $\w_0\in H^s(\R^d)$
 is determined by
\begin{equation}\label{rnax}
   \dv\,\w_0=0,\,\; \cu(r(S_0,0)\w_0)=\cu(r(S_0,0)\v_0),
\;\, r(S_0,0):= \lim_{\epsilon \rightarrow 0}r^\epsilon(S^\epsilon_0,0).
\end{equation}
The function $\hat \pi\in C([0,T]\times \R^d)$ satisfies $\nabla\hat \pi\in C([0,T],H^{s-1}(\R^d))$.
\end{thm}

\begin{proof}[Sketch of the proof of Theorem \ref{NTa}]
As explained before, the main step is to establish the uniform estimate \eqref{rnau}.
For this purpose, we define $\mathcal{M}_\epsilon(T)$ as follows
 \begin{align}
\mathcal{M}_\epsilon(T):=\,&  {\mathcal{N}_\epsilon(T)}^2
+\int_0^T\|\H^\epsilon\|_{s+1}^2d\tau,\label{rnbb}
\end{align}
where $\mathcal{N}_\epsilon (T)$ is defined by \eqref{nbbbb}. By arguments similar to
those used in the proof of Theorem \ref{prop}, one can obtain the desired estimate.
Indeed, the arguments are  easier since one can use the magnetic diffusion term
 to control the terms involving $\H$ in the momentum equations, and therefore we omit
 the details here.  \end{proof}

\medskip
 \noindent
{\bf Acknowledgements:}  The authors would like to thank Prof.
Fanghua Lin for suggesting this problem and for helpful discussions.
This work was partially done when Li  visited the Institute of
Applied Physics and Computational Mathematics in Beijing. He would
like to thank the institute for hospitality. Jiang was supported by
the National Basic Research Program under the Grant 2011CB309705
 and NSFC (Grant No. 40890154). Ju was supported by
NSFC (Grant No. 40890154, 11171035). Li was supported  by NSFC (Grant No.
10971094), PAPD, and the Fundamental Research Funds for the
Central Universities.


\end{document}